\documentclass[10pt]{article}
\usepackage[T1]{fontenc}
\usepackage{amsmath,amsfonts,amsthm,mathrsfs,amssymb,cite}
\usepackage{graphicx,appendix}
\usepackage{subfigure}
\usepackage{placeins}
\usepackage{color}
\usepackage{indentfirst}
\numberwithin{equation}{section}
\newcommand{\R}{{\mathbb R}}
\newcommand{\Z}{{\mathbb Z}}

\newcommand{\C}{{\mathbb C}}

\newcommand{\no}{\nonumber}
\newcommand{\be}{\begin{eqnarray}}
\newcommand{\ben}{\begin{eqnarray*}}
\newcommand{\en}{\end{eqnarray}}
\newcommand{\enn}{\end{eqnarray*}}

\newcommand{\pa}{\partial}

\newcommand{\ov}{\overline}

\newcommand{\G}{\Gamma}

\newcommand{\om}{\omega}

\newtheorem{theorem}{Theorem}[section]
\newtheorem{lemma}[theorem]{Lemma}

\newtheorem{remark}[theorem]{Remark}

\definecolor{rot}{rgb}{0.000,0.000,0.000}
\definecolor{rot1}{rgb}{0.000,0.000,0.000}
\begin{document}
\renewcommand{\theequation}{\arabic{section}.\arabic{equation}}
\begin{titlepage}
\title{\bf Boundary integral equation methods for the two dimensional fluid-solid interaction problem}


\author{
Tao Yin\ \ ({\sf taoyin\_cqu@163.com}) \ \\
 {\small College of Mathematics and Statistics, Chongqing University,
 Chongqing 401331, China }\\ \\
George C. Hsiao \ \ ({\sf  ghsiao@udel.edu})  \\
{\small Department of Mathematical Sciences, University of Delaware, Newark, DE 19716-2553, USA }\\ \\
Liwei Xu \ \ ({\sf xul@cqu.edu.cn}) \ \\
{\small College of Mathematics and Statistics, Chongqing University, Chongqing 401331, China} \\
{\small Institute of Computing and Data Sciences, Chongqing University, Chongqing 400044, China }
}
\date{}
\end{titlepage}
\maketitle
\vspace{.2in}

\begin{abstract}
This paper is concerned with boundary integral equation methods for solving the two-dimensional fluid-solid interaction problem. We reduce the problem to three differential systems of boundary integral equations via direct and indirect approaches. Existence and uniqueness results for variational solutions of  boundary integral equations equations are established. Since in  all these boundary variational formulations, the hypersingular boundary integral operator associated with the time-harmonic Navier equation is a dominated integral operator,  we also include a new regularization formulation for this hypersingular operator, which allows us to treat the hypersingular kernel by a wealkly singular kernel. Numerical examples are presented to verify and validate the theoretical results.

\vspace{.2in} {\bf Keywords:} Fluid-solid interaction problem, boundary integral equation method, Helmholtz equation, time-harmonic Navier equation.
\end{abstract}

\section{Introduction}

The fluid-solid interaction (FSI) problem can be simply described as an acoustic wave propagates in a  compressible fluid domain of infinite extent in which a bounded elastic body is immersed. The problem is to determine the scattered pressure field in the fluid domain as well as the displacement filed of the elastic body. It is of great importance in many fields of application including exploration seismology, oceanography,  and non-destructive testing, to name a few. Under the hypothesis of small amplitude oscillations both in the solid and fluid, the acoustic scattered pressure field and the elastic displacement satisfy the Helmholtz equation and time-harmonic Navier equation, respectively, together with appropriate transmission conditions across the fluid-solid interface. This problem encounters considerable mathematical challenges from both theoretical and computational points of views and has  been a subject of interest in both mathematical and engineering community for many years. In particular,  the unbounded domain in which the problem is imposed causes major difficulties  from computational point of view.  Our interest here is to develop efficient numerical methods for treating the two dimensional FSI problem. There are also some inverse FSI problems and FSI eigenvalue problems  being investigated in \cite{YHXZ16} and \cite{KSU14}, respectively.

One popular method to overcome the difficulty that the acoustic scattered wave propagates in an unbounded domain is known as the Dirichlet-to-Neumann (DtN) method (\cite{GHM09,YX,YRX}), that is, the original transmission problem is reduced to a boundary value problem by introducing a DtN mapping defined on an artificial boundary enclosing the elastic body inside. Another conventional numerical method is the coupling of the boundary element method (BEM) and the finite element method (FEM) (\cite{DSM12,DSM121,GMM07,GMM09,H94,HKR00,HKS89,MMS04}). Precisely, the BEM and FEM are employed for solving fields of the exterior acoustic wave and the interior elastic wave, respectively.

The boundary integral equation (BIE) methods for solving the scattering transmission problem including the acoustic transmission problem (\cite{CS85,HX11,KM88}), the elastic transmission problem (\cite{BLR14}), the electromagnetic transmission problem (\cite{BHPS03,CS88}) and the FSI problem (\cite{LM95}) have been extensively investigated for many years. One can derive the system of boundary integral equations (BIEs) equivalent with the original scattering problems by the direct method based on Green's formulation and the indirect method based on potential theory. In this paper, we derive three differential systems of BIEs for the solution of the FSI problem. For each system, we study the existence and uniqueness results for the weak solutions of corresponding variational equations. In addition to the so-called Jones frequency associated with the original transmission problem, it will be shown that only the first system of BIEs to be presented in Section $3$, for the purpose of uniqueness, need to exclude a spectrum of eigenvalues which is inherited from properties of boundary integral operators. Since all derived systems of BIEs are strongly elliptic, consisting of singular and hypersingular boundary integral operators,  appropriate regularization formulations are needed for the purpose of numerical computations. For the hypersingular boundary integral operator associated with Helmholtz equation, its expression can be transformed into one involving tangential rather than normal derivatives, see \cite{BM71,M49,M66} for details. With the help of the tangential G\"unter derivative (\cite{KGBB79}), variant representation of the hypersingular boundary integral operators associated with three dimensional Lam\'e equation is given in \cite{HW08}. In this paper, we will present an innovative and new regularization formulation for the hypersingular boundary integral operator associated with two dimensional elastodynamics and in the corresponding duality pairing form, only a weakly singular boundary integral operator is involved. Numerical results will be presented to illustrate efficiency of these systems of BIEs for  solutions of the FSI problem, and   accuracy of regularization formulation.

The remainder of the paper is organized as follows. We first describe the classical FSI problem in Section \ref{sec:2}. Using the direct and indirect approaches, and the Burton-Miller formulation based on direct method, we reduce the original problem to three different systems of coupled boundary integral equations in Section \ref{sec:3}, \ref{sec:4} and \ref{sec:5}, respectively. In each section, we also present the corresponding variational formulations of these systems, and carry out the uniqueness and existence analysis for the weak solution of the variational equations. In Section \ref{sec:6}, we present an innovative regularization formulation for the hypersingular boundary integral operator associated with the time-harmonic Navier equation and postpone the derivation in Appendix. The corresponding variational equations are reduced to discrete linear systems of equations by Galerkin boundary element method in Section \ref{sec:7}. In Section \ref{sec:8}, we present several numerical tests to confirm our theoretical results and verify the efficiency and accuracy of the Galerkin boundary element method. In closing the paper some conclusions and remarks for future work are presented in Section \ref{sec:9}.

\section{Statement of the problem}
\label{sec:2}
Let $\Omega\subset \R^2$ be a bounded, simply connected domain with sufficiently smooth boundary $\Gamma=\partial\Omega$,  and its exterior complement is denoted by $\Omega^c= \R^2\setminus\overline{\Omega}\subset \R^2$. The domain $\Omega$ is occupied by a linear and isotropic elastic solid, and $\Omega^c$ is filled with compressible, inviscid fluids. We denote by $\omega$ the frequency, $k=\omega/c$ the acoustic wave number, $c$ the speed of sound in the fluid,  $\rho$ the density of the solid and $\rho_f$ the density of the fluid. The problem we will solve is to determine the elastic displacement ${\bf u}$ in the solid and the acoustic scattered {pressure} field $p$ in the fluid {with a given} incident field $p^{inc}$. {The problem states} as follows: {\it Given $p^{inc}$, find ${\bf u}\in (C^2(\Omega)\cap C^1(\overline\Omega))^2$ and $p\in C^2(\Omega^c)\cap C^1(\overline{\Omega^c})$ satisfying :}
\be
\label{Eq:Navier}
\Delta^{*}{\bf u} +   \rho \omega^2{\bf u} &=& {\bf 0}\quad \text{in}\quad \Omega, \\
\label{Eq:Helmholtz}
\Delta p + k^2p &=& 0\quad\, \text{in}\quad \Omega^c,
\en
{together with the transmission conditions}
\begin{eqnarray}
\label{Eq:TransCond1}
\eta{\bf u}\cdot {\bf n} &=& \frac{\partial}{\partial n}(p+p^{inc}) \quad \text{on}\quad \Gamma,\\
\label{Eq:TransCond2}
{\bf t} &=& -{\bf n}(p + p^{inc}) \quad \text{on}\quad \Gamma,
\end{eqnarray}
 and the Sommerfeld radiation condition
\be
\label{Eq:RadiationCond}
\lim_{r \to \infty} r^{\frac{1}{2}}\left(\frac{\partial p}{\partial r}-ikp\right) = 0,\quad r=|x|.
\en
Here, $\partial/\partial n$ is the normal derivative on $\Gamma$ (here and in the sequel, ${\bf n}$ is always the outward unit normal to the boundary), $i=\sqrt{-1}$ is the imaginary unit, $x=(x_1,x_2)\in \R^2$, $\eta=\rho_f\omega^2$. $\Delta^{*}$ is the Lam\'{e} operator defined by
\ben
\Delta^*\, {:=} \, \mu\Delta + (\lambda + \mu) \, grad \,div,
\enn
where, $\lambda,\mu$ are Lam\'e constants such that  {$\mu\,>\,0$ and $\lambda\, + \mu > 0$.} In addition, ${\bf t} = {\bf T}{\bf u}$ and ${\bf T}$ is the traction operator on the boundary defined by
\ben
{\bf T}{\bf u} \, {:= \lambda \,(div {\bf u})\, {\bf n} + 2\mu\frac{\partial{\bf u}}{\partial n} } + \mu\, {\bf n}\times  {curl}\, {\bf u}.
\enn

\begin{figure}[htbp]
\centering
\includegraphics[scale=1.5]{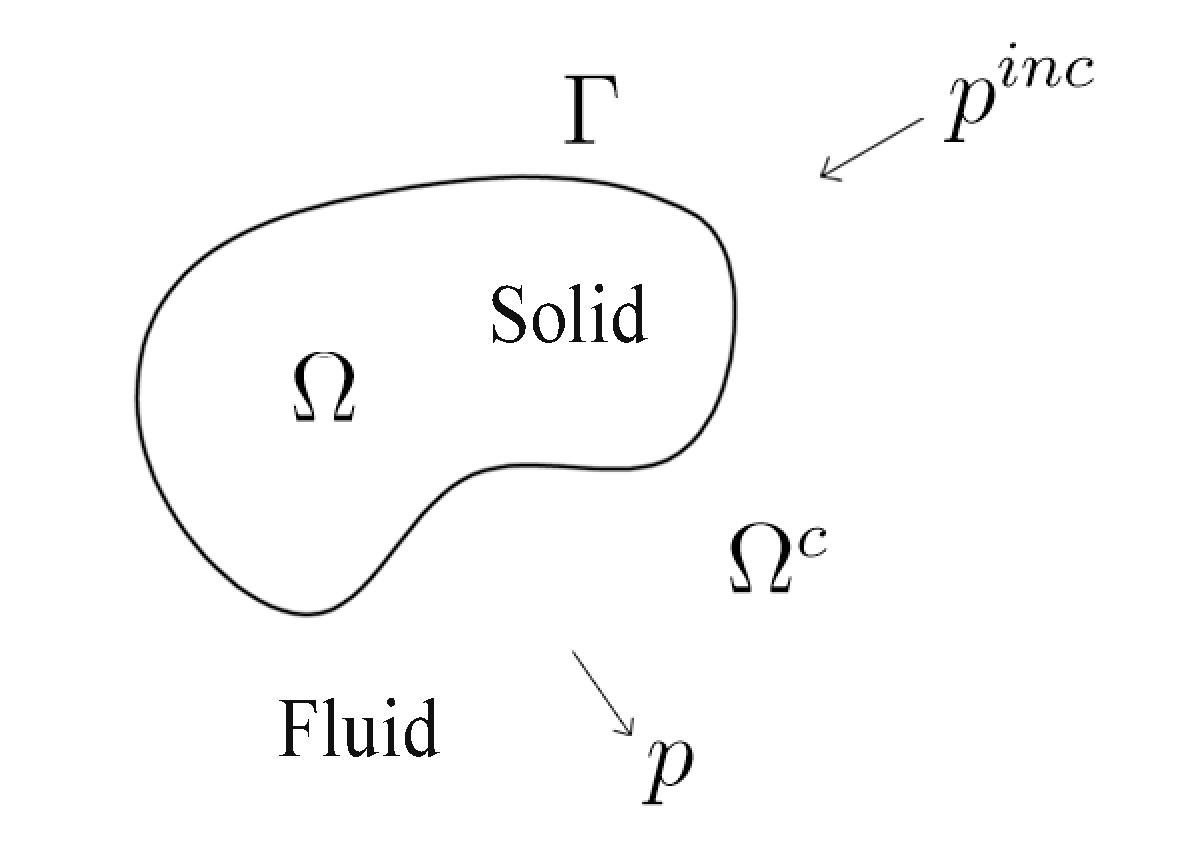}
\caption{Boundary value problem \eqref{Eq:Navier}--\eqref{Eq:RadiationCond}.}
\end{figure}

It is known (\cite{JDS83}) that, for certain geometries and some frequencies $\omega$, the problem \eqref{Eq:Navier}--\eqref{Eq:RadiationCond} is not always uniquely solvable due to the occurrence of so-called traction free oscillations. These $\omega$ are also known as the Jones frequencies which are inherent to the original model. {We state without proof of the following result}:
\begin{theorem} \label{th:2.1}
If the surface $\Gamma$ and the material parameters $(\mu,\lambda,\rho)$ are such that there are no traction free solutions, the boundary value problem \eqref{Eq:Navier}--\eqref{Eq:RadiationCond} has at most one solution, { provided $Im\, k = 0$.} Here, we call a nontrivial ${\bf u}_0$ a traction free solution if it solves
\ben
\Delta^{*}{\bf u}_0 + \rho\omega^2{\bf u}_0 &=& {\bf 0} \quad\mbox{in}\quad \Omega,  \\
{\bf T}{\bf u}_0 &=&{\bf 0}\quad \mbox{on}\quad\Gamma, \\
{\bf u}_0\cdot{\bf n} &=& 0\quad\mbox{on}\quad\Gamma.
\enn
\end{theorem}
{The proof is given in \cite{HKR00} and is based on a standard uniqueness result for the transmission problem in scattering below.
\begin{lemma} \label{le:2.2}
If $({\bf u}, p) $ is a classical solution of the corresponding homogeneous problem of   \eqref{Eq:Navier}-\eqref{Eq:RadiationCond} for $Im\,k = 0$ with $p^{inc} = 0 $, then $ p \equiv 0.$
\end{lemma}

Through out the paper, we always assume that $\omega>0$, $k>0$.}
\section{Direct method} \label{sec:3}
In this section, the transmission problem \eqref{Eq:Navier}-\eqref{Eq:RadiationCond} is reduced to a system of coupled boundary integral equations consisting of four basic boundary integral operators based on direct approach. This system of boundary integral equations is then converted into its weak formulation for the study of uniqueness and existence of the weak solution.

\subsection{Boundary integral equations} \label{sec:3.1}
Classical solutions ${\bf u}$ and $p$ can be represented by boundary integral equations via the Green's representation formula and the fundamental displacement tensor ${\bf E}(x,y)$ of the time-harmonic Navier equation \eqref{Eq:Navier}  as well as the fundamental solution $\gamma_k(x,y)$ of the Helmholtz equation \eqref{Eq:Helmholtz} in $\R^2$.  In terms of the classical analysis, $\bf{u}$ and $p$ read
\be
\label{Eq:DirectBRF1}
{\bf u}(x)  &=&  \int_{\Gamma}{\bf E}(x,y){\bf t}(y)\,ds_y - \int_{\Gamma}({\bf T}_y{\bf E}(x,y))^\top{\bf u}(y)\,ds_y, \quad \forall\,x\in\Omega,\\
\label{Eq:DirectBRF2}
p(x)  &=&  \int_{\Gamma}\frac{\partial \gamma_k }{\partial n_y}(x,y)p(y)\,ds_y - \int_{\Gamma} \gamma_k(x,y)\frac{\partial p}{\partial n_y}(y)\,ds_y, \quad\forall\,x\in \Omega^c,
\en
respectively. The fundamental solution of the Helmholtz equation \eqref{Eq:Helmholtz} in $\R^2$ takes the form
\be
\gamma_k(x,y) = \frac{i}{4}H_0^{(1)}(k|x-y|),\quad x\ne y.   \label{Eq:HelmholtzFS}
\en
Here, $H_0^{(1)}(\cdot)$ is the first kind  Hankel function of order $0$. We {denote by}
\ben
k_s = \omega\sqrt{\rho/\mu} \;\; \mbox{{and}} \; \; k_p = \omega \sqrt{\rho/(\lambda + 2\mu)}.
\enn
{respectively, the shear (or transverse) and the compressional (or longitudinal) elastic  wave numbers}. Then the fundamental displacement tensor ${\bf E}(x,y)$ can be written as
\be \label{Eq:NavierFS}
{\bf E}(x,y) = \frac{1}{\mu}\gamma_{k_s}(x,y){\bf I} + \frac{1}{\rho\omega^2}\nabla_x\nabla_x\left[\gamma_{k_s}(x,y) - \gamma_{k_p}(x,y)\right],\quad x\ne y.
\en
where  ${\bf I}$ denotes the identity matrix. Now, letting  $x$ in equations  \eqref{Eq:DirectBRF1}--\eqref{Eq:DirectBRF2}  approach to the boundary $\Gamma$ and applying the jump conditions, we obtain the corresponding boundary integral equations on $\Gamma$
\be
{\bf u}(x) &=& V_{s}{\bf t}(x) + \left(\frac{1}{2}I - K_{s}\right){\bf u}(x),\quad x\in\Gamma,  \label{Eq:ElasticBIE}\\
p(x) &=& \left(\frac{1}{2}I + K_f\right)p(x) - V_{f}\frac{\partial p}{\partial n}(x),\quad x\in\Gamma.
\label{Eq:FluidBIE}
\en
Operating with the traction operator on \eqref{Eq:DirectBRF1}, computing the norm derivative for both sides of \eqref{Eq:DirectBRF2} and taking the limits as $x\to\Gamma$,  and applying the jump relations,  we are led to the following additional boundary integral equations on $\Gamma$
\be \label{Eq:ElasticBIE1}
{\bf t}(x) &=& \left(\frac{1}{2}I + K_s^{'}\right){\bf t}(x) + W_s{\bf u}(x),\quad x\in\Gamma, \\
\label{Eq:FluidBIE1}
\frac{\partial p}{\partial n}(x) &=& \left(\frac{1}{2}I - K_f^{'}\right)\frac{\partial p}{\partial n}(x) - W_fp(x),\quad x\in\Gamma.
\en
In boundary integral equations \eqref{Eq:ElasticBIE}--\eqref{Eq:FluidBIE1}, $I$ is the identity operator, the boundary integral operators related with the fluid are defined by
\ben
V_f\frac{\partial p}{\partial n}(x) &=& \int_{\Gamma} \gamma_k(x,y)  \frac{\partial p}{\partial n_y}(y)\,ds_y,\quad x\in\Gamma,\\
K_f p(x) &=& \int_{\Gamma}\frac{\partial \gamma_k}{\partial n_y}(x,y)p(y)\,ds_y, \quad x\in\Gamma,\\
K_f^{'}\frac{\partial p}{\partial n}(x) &=& \int_{\Gamma} \frac{\partial \gamma_k}{\partial n_x}(x,y) \frac{\partial p}{\partial n}(y)\,ds_y,\quad x\in\Gamma,\\
W_f p(x) &=& -\frac{\partial}{\partial n_x}\int_{\Gamma}\frac{\partial \gamma_k}{\partial n_y}(x,y)p(y)\,ds_y,\quad x\in\Gamma,
\enn
and the boundary integral operators for the elasticity are defined by
\ben
V_s{\bf t}(x) &=& \int_{\Gamma}{\bf E}(x,y) {\bf t}(y)\,ds_y,\quad x\in\Gamma,\\
K_s{\bf u}(x) &=& \int_{\Gamma}({\bf T}_y{\bf E}(x,y))^\top{\bf u}(y)\,ds_y,\quad x\in\Gamma,\\
K_s^{'}{\bf t}(x) &=& \int_{\Gamma}{\bf T}_x{\bf E}(x,y){\bf t}(y)\,ds_y,\quad x\in\Gamma,\\
W_s{\bf u}(x)  &=& -{\bf T}_x\int_{\Gamma}({\bf T}_y{\bf E}(x,y))^\top{\bf u}(y)\,ds_y,\quad x\in\Gamma.
\enn
In the above equations, neglect of  subindex  $s$ and $f$ on the operators { associated with} the solid and fluid, $V$, $K$, $K^{'}$ and $W$ are termed respectively, the single-layer, double-layer, transpose of double-layer and hypersingular boundary integral operators.

We combine boundary integral equations  \eqref{Eq:ElasticBIE1}--\eqref{Eq:FluidBIE1} together with transmission conditions  \eqref{Eq:TransCond1}--\eqref{Eq:TransCond2} to obtain a system of coupled boundary integral equations for a pair of unknown functions ${\bf u}$ and $p$.  Therefore, to eliminate $ {\bf t}$ and $ \partial p/ \partial n $,  we arrive at for $x\in\Gamma$
\be
W_s{\bf u}(x) + \left(\frac{1}{2}I - K_s^{'}\right)p{\bf n}(x) &=& \left(- \frac{1}{2}I + K_s^{'} \right)p^{inc}{\bf n}(x) := f_1,   \label{Eq:BIE1}\\
\eta\left(\frac{1}{2}I + K_f^{'}\right)({\bf u}\cdot {\bf n})(x) + W_f p(x)  &=& \left(\frac{1}{2}I + K_f^{'}\right)\frac{\partial p^{inc}}{\partial n}(x) := f_2. \label{Eq:BIE2}
\en
Since we assume that the interface $\Gamma$ is sufficiently smooth, the boundary integral operators are continuous mappings for the indicated function spaces below (\cite{HKR00, HW08})
\ben
W_s\quad&:&\quad (H^{s+1}(\Gamma))^2  \mapsto  (H^s(\Gamma))^2,\\
K_s^{'} \quad&:& \quad (H^s(\Gamma))^2 \mapsto (H^s(\Gamma))^2,\\
K_f^{'} \quad&:& \quad H^{s}(\Gamma) \mapsto  H^{s+1}(\Gamma),\\
W_f \quad&:&\quad H^{s+1}(\Gamma) \mapsto  H^s(\Gamma).
\enn

The unique solvability of \eqref{Eq:BIE1}--\eqref{Eq:BIE2} is given in the following theorem.
\begin{theorem}
\label{Theorem3.11}
If\\
(a) \, the surface $\Gamma$ and the material parameters $(\mu,\lambda,\rho)$ are such that there are no traction free solutions,\\
(b) \, $-k^2$ is not an eigenvalue of the interior Neumann problem for the Laplacian, \\
the system of boundary integral equations \eqref{Eq:BIE1}--\eqref{Eq:BIE2} is uniquely solvable.
\end{theorem}
\begin{proof}
The proof follows similarly as the procedure given in \cite{LM95}. It is sufficient to prove that the corresponding homogeneous system has only the trivial solution. Suppose that $({\bf u}_0,p_0)$ is a solution of the corresponding homogeneous system of \eqref{Eq:BIE1}--\eqref{Eq:BIE2}. Now, let
\ben
{\bf u}_e(x)  &=&  -\int_{\Gamma}{\bf E}(x,y)({\bf n}p_0)(y)\,ds_y - \int_{\Gamma}({\bf T}_y{\bf E}(x,y))^\top{\bf u}_0(y)\,ds_y, \quad \forall\,x\in\Omega^c,\\
p_i(x)  &=&  \int_{\Gamma}\frac{\partial \gamma_k }{\partial n_y}(x,y)p_0(y)\,ds_y - \eta\int_{\Gamma} \gamma_k(x,y)({\bf u}_0\cdot {\bf n})(y)\,ds_y, \quad\forall\,x\in \Omega.
\enn
Taking the limit $x\to\Gamma$, and making use of jump relations of the single- and double-layer potentials, we arrive at boundary integral equations:
\be
\label{Eq:Th31BIE1}
({\bf Tu}_e)(x) &=& \left(\frac{1}{2}I - K_s^{'}\right)({\bf n}p_0)(x) + W_s{\bf u}_0(x),\quad x\in\Gamma,\\
\label{Eq:Th31BIE2}
\frac{{\partial p_i}}{\partial n}(x) &=& -W_fp_0(x) - \eta\left(\frac{1}{2}I + K_f^{'}\right)({\bf n}\cdot {\bf u}_0)(x),\quad x\in\Gamma.
\en
Then the homogeneous form of \eqref{Eq:BIE2} and \eqref{Eq:Th31BIE2} implies that $\partial p_i/\partial n=0$ on $\Gamma$. Under the assumption (b) we know that $p_i=0$ in $\Omega$. In particular, $p_i=0$ on $\Gamma$. Then we have
\be
\label{Eq:Th31BIE3}
\left(-\frac{1}{2}I + K_f\right)p_0(x) - \eta V_{f}({\bf n}\cdot {\bf u}_0)(x)=0,\quad x\in\Gamma.
\en
In addition, since ${\bf u}_e$ satisfies the exterior elastic scattering problem in $\Omega^c$ with vanishing Neumann boundary data on $\Gamma$, it follows that ${\bf u}_e=0$ in $\Omega^c$. In particular, ${\bf u}_e=0$ on $\Gamma$. Then we have
\be
\label{Eq:Th31BIE4}
V_{s}({\bf n}p_0)(x) + \left(\frac{1}{2}I + K_{s}\right){\bf u}_0(x)=0,\quad x\in\Gamma.
\en
Now, let
\be
\label{Eq:Th31BIE5}
{\bf u}_i(x)  &=&  -\int_{\Gamma}{\bf E}(x,y)({\bf n}p_0)(y)\,ds_y - \int_{\Gamma}({\bf T}_y{\bf E}(x,y))^\top{\bf u}_0(y)\,ds_y, \quad \forall\,x\in\Omega^c,\\
\label{Eq:Th31BIE6}
p_e(x)  &=&  \int_{\Gamma}\frac{\partial \gamma_k }{\partial n_y}(x,y)p_0(y)\,ds_y - \eta\int_{\Gamma} \gamma_k(x,y)({\bf u}_0\cdot {\bf n})(y)\,ds_y, \quad\forall\,x\in \Omega.
\en
Evaluating the jump on the boundary $\Gamma$, we obtain from \eqref{Eq:Th31BIE3} and \eqref{Eq:Th31BIE6} that
\ben
\frac{{\partial p_e}}{\partial n}-\frac{{\partial p_i}}{\partial n}=\eta u_0\cdot{\bf n}\quad\mbox{and}\quad p_e=p_0 \quad\mbox{on}\quad\Gamma.
\enn
On the other hand, we derive from \eqref{Eq:Th31BIE4} and \eqref{Eq:Th31BIE5} that
\ben
{\bf Tu}_i-{\bf Tu}_e=-{\bf n}p_0\quad\mbox{and}\quad {\bf u}_i={\bf u}_0 \quad\mbox{on}\quad\Gamma.
\enn
Thus, $({\bf u}_i,p_e)$ is a solution of the homogeneous boundary value problem of \eqref{Eq:Navier}-\eqref{Eq:RadiationCond}. This further implies that ${\bf u}_i=0$ in $\Omega$ and $p_e=0$ in $\Omega^c$ under assumption (a) which further leads to $p_0=0$ and ${\bf u}_0=0$ on $\Gamma$. This completes the proof.
\end{proof}

\subsection{Weak formulation}
\label{sec3.2}
Now, we consider the weak formulation for the system of boundary integral equations \eqref{Eq:BIE1}--\eqref{Eq:BIE2}. We assume that
\ben
{\bf u} \in (H^1(\Omega))^2 \quad \text{and}\quad p\in H_{loc}^1(\Omega^c)
\enn
with traces
\ben
{\bf u}|_{\Gamma}\in (H^{1/2}(\Gamma))^2 \quad \text{and} \quad p|_{\Gamma}\in H^{1/2}(\Gamma),
\enn
respectively.  Then the standard weak formulation takes the form:\quad {\it Given $p^{inc}$ and $\partial p^{inc}/\partial n$, find $({\bf u}, p)\in \mathcal{H}(\Gamma)=(H^{1/2}(\Gamma))^2\times H^{1/2}(\Gamma)$ satisfying}
\be
A({\bf u},p; {\bf v},q)  =  F({\bf v},q), \quad\forall \,({\bf v},q)\in  \mathcal{H}(\Gamma),
\label{Eq:WeakForm1}
\en
where the sesquilinear form $A(\cdot\,;\, \cdot): \mathcal{H}(\Gamma)\times \mathcal{H}(\Gamma)\mapsto \R$ is defined by
\be \label{Eq:SesquiForm1}
A({\bf u},p; {\bf v},q): &=&  \langle W_s{\bf u}, {\bf v}\rangle + \left\langle \left(\frac{1}{2}I - K_s^{'}\right)p{\bf n}, {\bf v}\right\rangle \nonumber\\
&+& \eta\left\langle \left(\frac{1}{2}I + K_f^{'}\right)({\bf u}\cdot {\bf n}),q\right\rangle + \langle W_fp,q\rangle,
\en
and the linear functional $F({\bf v},q)$ on $\mathcal{H}(\Gamma)$ is defined  by
\ben
F({\bf v},q) = \langle f_1, {\bf v}\rangle + \langle f_2, q\rangle.
\enn
Here and in the sequel, $\langle\cdot,\cdot\rangle$ is the $L^2$ duality pairing between $H^{-1/2}(\Gamma)$ and $H^{1/2}(\Gamma)$, or $(H^{-1/2}(\Gamma))^2$ and $(H^{1/2}(\Gamma))^2$.  In order to obtain the existence of a solution of the variational equation \eqref{Eq:WeakForm1}, we need the next two theorems.
\begin{theorem} \label{Theorem3.1}
The sesquilinear form \eqref{Eq:SesquiForm1} satisfies a G\aa rding's inequality in the form
\be
\label{Eq:Garding1}
\mbox{Re}\,\{ A({\bf u},p; {\bf u},p)\} &\ge& \alpha\left(\|{\bf u}\|_{(H^{1/2}(\Gamma))^2}^2 + \|p\|_{H^{1/2}(\Gamma)}^2\right) \nonumber\\
&-& \beta\left(\|{\bf u}\|_{(H^{1/2-\epsilon}(\Gamma))^2}^2 + \|p\|_{H^{1/2-\epsilon}(\Gamma)}^2\right),
\en
for all $({\bf u}, p)\in \mathcal{H}(\Gamma)$ where $\alpha>0$ , $\beta\ge 0$ and $0<\epsilon<1/2$ are all constants.
\end{theorem}
\begin{proof}
In \eqref{Eq:SesquiForm1}, we set the test functions $({\bf v},q)$ to be $({\bf u},p)$ and obtain
\ben
&\quad&A({\bf u},p; {\bf u},p) \\
&=& \langle W_s{\bf u}, {\bf u}\rangle + \left\langle \left(\frac{1}{2}I - K_s^{'}\right)p{\bf n}, {\bf u}\right\rangle + \eta\left\langle \left(\frac{1}{2}I + K_f^{'}\right)({\bf u}\cdot {\bf n}),p\right\rangle + \langle W_fp,p\rangle.
\enn
We notice first that $W_f$ satisfies a G\aa rding's inequality in the form (\cite{HW08})
\be
\mbox{Re}\,\{\langle W_fp,p\rangle\} \ge \alpha_1\|p\|_{H^{1/2}(\Gamma)}^2 - \beta_1\|p\|_{H^{1/2-\epsilon}(\Gamma)}^2 \label{Eq:Theorem3.1-Eq1}
\en
for some constants $\alpha_1 > 0$,  $\beta_1\ge 0$ and $0<\epsilon<1/2$.  From the estimates in {\cite{HKR00}} we know that
 \be  \label{Eq:Theorem3.1-Eq2}
\mbox{Re}\,\{\langle W_s{\bf u}, {\bf u}\rangle\} \ge \alpha_2\|{\bf u}\|_{(H^{1/2}(\Gamma))^2}^2 - \beta_2\|{\bf u}\|_{(H^{1/2-\epsilon}(\Gamma))^2}^2
\en
for some constants $\alpha_2 > 0$, $\beta_2\ge 0$ and $0<\epsilon<1/2$.
Furthermore, we have
\ben
\left|\left\langle \left(\frac{1}{2}I + K_f^{'}\right)({\bf u}\cdot {\bf n}),p\right\rangle\right|&\le&  \left\|\left(\frac{1}{2}I + K_f^{'}\right)({\bf u}\cdot {\bf n})\right\|_{H^{0}(\Gamma)}\|p\|_{H^{0}(\Gamma)}\nonumber\\
&\le& c \|{\bf u}\|_{(H^{0}(\Gamma))^2}\|p\|_{H^{0}(\Gamma)}\nonumber\\
&\le& c \left(\|{\bf u}\|_{(H^{1/2-\epsilon}(\Gamma))^2}^2 + \|p\|_{H^{1/2-\epsilon}(\Gamma)}^2\right)
\enn
which follows immediately by
\be \label{Eq:Theorem3.1-Eq3}
\mbox{Re}\,\left\{\left\langle \left(\frac{1}{2}I + K_f^{'}\right)({\bf u}\cdot {\bf n}),p\right\rangle\right\} \ge -\beta_3\left(\|{\bf u}\|_{(H^{1/2-\epsilon}(\Gamma))^2}^2 + \|p\|_{H^{1/2-\epsilon}(\Gamma)}^2\right),
\en
where $\beta_3\ge 0$ and $0<\epsilon<1/2$ are all constants.  Due to the same argument, we also have
\be \label{Eq:Theorem3.1-Eq4}
\mbox{Re}\,\left\{\left\langle \left(\frac{1}{2}I - K_s^{'}\right)(p{\bf n}),{\bf u}\right\rangle\right\} \ge -\beta_4\left(\|{\bf u}\|_{(H^{1/2-\epsilon}(\Gamma))^2}^2 + \|p\|_{H^{1/2-\epsilon}(\Gamma)}^2\right)
\en
for some constants $\beta_4\ge 0$ and $0<\epsilon<1/2$ . Therefore, the combination of inequalities \eqref{Eq:Theorem3.1-Eq1}--\eqref{Eq:Theorem3.1-Eq4} gives the G\aa rding's inequality \eqref{Eq:Garding1} immediately. This completes the proof.
\end{proof}

Now, the existence result follows immediately from the Fredholm's Alternative: uniqueness implies existence. Therefore, we have the following theorem.

\begin{theorem} \label{Theorem3.3}
{Under the assumptions {\em (a)} and {\em (b)} in Theorem \ref{Theorem3.11}, the variational equation \eqref{Eq:WeakForm1} admits a unique solution $({\bf u},p)\in  \mathcal{H}(\Gamma)$}.
\end{theorem}

\section{Indirect method}
\label{sec:4}
We now employ the indirect method based on the potential layers to derive a  system of coupled boundary integral equations for the solution of the problem \eqref{Eq:Navier}--\eqref{Eq:RadiationCond}. Uniqueness and existence results are established for the weak solution in appropriate Sobolev spaces.
\subsection{Boundary integral equations}
\label{sec:4.1}
In terms of the representation formulas \eqref{Eq:DirectBRF1}--\eqref{Eq:DirectBRF2} and potential theory, we may seek the solution of problem \eqref{Eq:Navier}--\eqref{Eq:RadiationCond} in the form of combined single- and double-layer potentials
\be
\label{Eq:IndirectBRF1}
{\bf u}(x) &=& S_s(-{\bf n}\psi )(x) - D_s({\bf v}), \quad\forall\,  x\in\Omega,\\
\label{Eq:IndirectBRF2}
p(x) &=& D_f(\psi) - \eta S_f({\bf v}\cdot{\bf n} ), \quad\forall\, x\in\Omega^c,
\en
where  ${\bf v}$ and $\psi$ are two unknown continuous density functions defined on the space $(H^{1/2}(\Gamma))^2$ and $H^{1/2}(\Gamma)$, respectively. $S_s$ and $D_s$ are the standard single- and double-layer potentials defined on the solid region $\Omega$. Similarly, $S_f$ and $D_f$ are single- and double-layer potentials on the fluid domain $\Omega^c$.  Taking the limit $x\to \Gamma$  directly for \eqref{Eq:IndirectBRF1}, and operating with the traction operator on \eqref{Eq:IndirectBRF1} and then taking the limit $x\to \Gamma$,  we arrive at, by the jump conditions,
\be
\label{Eq:IDBIE1}
{\bf u}(x) &=& V_s(-{\bf n}\psi)(x) + \left(\frac{1}{2}I - K_s\right){\bf v}(x),\quad\forall\,x\in\Gamma,\\
\label{Eq:IDBIE2}
{\bf t}(x) &=& \left(\frac{1}{2}I + K_s^{'}\right)(-{\bf n}\psi)(x) + W_s{\bf v}(x),\quad\forall\,x\in\Gamma.
\en
Taking  the limit $x\to \Gamma$  directly for \eqref{Eq:IndirectBRF2}, and computing the normal derivative of \eqref{Eq:IndirectBRF2} and taking the limit $x\to \Gamma$,  we obtain, by the jump conditions,
\be
\label{Eq:IDBIE3-1}
p(x) &=& \left(\frac{1}{2}I + K_f\right)\psi(x) - \eta V_f({\bf n}\cdot {\bf v})(x),\quad \forall\,x\in\Gamma,\\
\label{Eq:IDBIE4-1}
\frac{\partial p}{\partial n}(x) &=& -W_f\psi(x) + \eta\left(\frac{1}{2}I - K_f^{'}\right)({\bf n}\cdot {\bf v})(x),\quad\forall\,x\in\Gamma.
\en
In the boundary integral equations \eqref{Eq:IDBIE1}--\eqref{Eq:IDBIE4-1}, boundary integral operators are the same as defined in Section \ref{sec:3}. Now, we start with the boundary integral equations \eqref{Eq:IDBIE1}  and \eqref{Eq:IDBIE2} consisting of the hyper-singular boundary integral operators and utilize the transmission conditions \eqref{Eq:TransCond1} and \eqref{Eq:TransCond2} to derive a system of coupled boundary integral equations with unknown vector  $({\bf v},\psi)$.  Taking the dot-product with ${\bf n}$ for both sides of \eqref{Eq:IDBIE1} and using \eqref{Eq:TransCond1} and \eqref{Eq:IDBIE4-1}, we have
\be
\label{Eq:IDBIE3}
&\quad&{\bf n}\cdot  V_s(-{\bf n}\psi)  + {\bf n}\cdot\left(\frac{1}{2}I - K_s\right){\bf v}\nonumber\\
&=& \frac{1}{\eta}\left\{ -W_f\psi  + \eta\left(\frac{1}{2}I - K_f^{'}\right)({\bf n}\cdot {\bf v})\right\} + \frac{1}{\eta}\frac{\partial p^{inc}}{\partial n}, \quad x\in\Gamma.
\en
Similarly, beginning with \eqref{Eq:IDBIE2} and using \eqref{Eq:TransCond2} and \eqref{Eq:IDBIE3-1},  we are led to
\be
\label{Eq:IDBIE4}
&\quad& \left(\frac{1}{2}I + K_s^{'}\right)(-{\bf n}\psi)  + W_s{\bf v} \nonumber\\
&=&-{\bf n}p^{inc} - \left\{{\bf n}\left(\frac{1}{2}I + K_f\right)\psi  - \eta{\bf n}V_f({\bf n}\cdot {\bf v})\right\}, \quad x\in\Gamma.
\en
By the combination of equations \eqref{Eq:IDBIE3} and \eqref{Eq:IDBIE4}, we immediately obtain a system of coupled boundary integral equations on $\Gamma$ as below
\be
\label{Eq:IDBIE5}
 W_s{\bf v} +  {\bf n} K_f\psi  - K_s^{'} ({\bf n}\psi) - \eta{\bf n}V_f({\bf n}\cdot {\bf v}) &=& -{\bf n}p^{inc}:=g_1,\quad x\in\Gamma,\\
\label{Eq:IDBIE6}
 \frac{1}{\eta} W_f\psi +  K_f^{'}({\bf n}\cdot {\bf v}) -{\bf n}\cdot K_s{\bf v}  - {\bf n}\cdot  V_s({\bf n}\psi)  &=&\frac{1}{\eta}\frac{\partial p^{inc}}{\partial n}:=g_2,\quad x\in\Gamma.
\en
It can be seen that  \eqref{Eq:IndirectBRF1}  and \eqref{Eq:IndirectBRF2} define a solution of the fluid-solid interaction problem \eqref{Eq:Navier}--\eqref{Eq:RadiationCond} if  $({\bf v},\psi)$  solves the system of boundary integral equations \eqref{Eq:IDBIE5}--\eqref{Eq:IDBIE6}.  Therefore, the unique  solvability of \eqref{Eq:Navier}--\eqref{Eq:RadiationCond} can be proved by showing that \eqref{Eq:IDBIE5}--\eqref{Eq:IDBIE6} is uniquely solvable. {First, we need the following result.}
\begin{theorem}
\label{Theorem4.1}
If the surface $\Gamma$ and the material parameter $(\mu,\lambda,\rho)$ are such that there are no traction free solutions, the system of boundary integral equations \eqref{Eq:IDBIE5}--\eqref{Eq:IDBIE6} is uniquely solvable.
\end{theorem}
\begin{proof}
It is sufficient to prove that the corresponding homogeneous system
has only the trivial solution. {Suppose $ ({\bf v}_0,\psi_0)$  is a solution  of the corresponding homogeneous system of \eqref{Eq:IDBIE5}--\eqref{Eq:IDBIE6}. Now, let
\be
\label{Eq:FluidBRFInside}
p(x) &=& D_f(\psi_0) - \eta S_f({\bf v_0}\cdot{\bf n} ), \quad x\in\Omega,\\
\label{Eq:SolidBRFOutside}
{\bf u}(x) &=& S_s(-{\bf n}\psi_0 )(x) - D_s({\bf v_0}), \quad  x\in\Omega^c.
\en
Taking the limit $x\to\Gamma$, and making use of jump relations of the single- and double-layer potentials, we arrive at boundary integral equations:
\be
\label{Eq:Th41BIE1}
{\bf u}(x) &=& V_s(-{\bf n}\psi_0)(x) - \left(\frac{1}{2}I + K_s\right){\bf v_0}(x),\quad x\in\Gamma,\\
\label{Eq:Th41BIE2}
{\bf t}(x) &=& \left(-\frac{1}{2}I + K_s^{'}\right)(-{\bf n}\psi_0)(x) + W_s{\bf v_0}(x),\quad x\in\Gamma,\\
\label{Eq:Th41BIE3}
p(x) &=& \left(-\frac{1}{2}I + K_f\right)\psi_0(x) - \eta V_f({\bf n}\cdot {\bf v_0})(x),\quad  x\in\Gamma,\\
\label{Eq:Th41BIE4}
\frac{{\partial p}}{\partial n}(x) &=& -W_f\psi_0(x) - \eta\left(\frac{1}{2}I + K_f^{'}\right)({\bf n}\cdot {\bf v_0})(x),\quad x\in\Gamma.
\en
Combinations of \eqref{Eq:Th41BIE2} and \eqref{Eq:Th41BIE3}, and  \eqref{Eq:Th41BIE1} and \eqref{Eq:Th41BIE4} yield, respectively
\ben
{\bf t} + p{\bf n}& =&  W_s{\bf v_0} +  {\bf n} K_f\psi_0  - K_s^{'} ({\bf n}\psi_0) - \eta{\bf n}V_f({\bf n}\cdot {\bf v_0}) =   {\bf 0}\quad\mbox{on}\quad \Gamma,\\
{\bf u} \cdot{\bf n} - \frac{1}{\eta}\frac{\partial p }{\partial n} &= & \frac{1}{\eta} W_f\psi_0 +  K_f^{'}({\bf n}\cdot {\bf v_0}) -{\bf n}\cdot K_s{\bf v_0}  - {\bf n}\cdot  V_s({\bf n}\psi_0) = 0\quad\mbox{on}\quad \Gamma,
\enn
since $({\bf v_0}, \psi_0) $  is a solution of the corresponding homogeneous system of \eqref{Eq:IDBIE5}--\eqref{Eq:IDBIE6}.
 }
Hence ${\bf u}$ and $p$ solve the following homogeneous fluid-solid interaction problem {consisting of}
\be
\label{Eq:Helmholtz4.1}
\Delta p + k^2p &=& 0 \quad\quad\quad\;\;\text{in}\quad \Omega,\\
\label{Eq:Navier4.1}
\Delta^*{\bf u} + \rho\omega^2{\bf u} &=& {\bf 0}\quad\quad\quad\;\;\text{in}\quad\Omega^c,\\
\label{Eq:TransCond4.1-1}
{\bf t} &=& -p{\bf n}\quad\quad\;\text{on}\quad\Gamma,\\
\label{Eq:TransCond4.1-2}
{\bf u}\cdot{\bf n} &=& \frac{1}{\eta}\frac{\partial p}{\partial n}\quad\quad\text{on}\quad\Gamma
\en
{together with elastic radiation conditions (\cite{HW89,LM95}) for {\bf u} }  given in terms of the pressure wave ${\bf u}_p$ and the shear wave ${\bf u}_s$ associated with the wave numbers $k_p$ and $k_s$, respectively. It follows from the Green's formulation and the transmission conditions \eqref{Eq:TransCond4.1-1}--\eqref{Eq:TransCond4.1-2}that
\be
\label{Eq:4.1-relation}
\int_{\Gamma_a}{\bf u}\cdot{\bf T}\ov{\bf u}\,ds &=& \int_\Gamma{\bf u}\cdot{\bf T}\ov{\bf u}\,ds+a({\bf u},{\bf u})\no\\
&=& -\int_\Gamma{\bf u}\cdot{\bf n}\ov{p}\,ds+a({\bf u},{\bf u})\no\\
&=& -\frac{1}{\eta}\int_\Gamma\frac{\pa p}{\pa n}\ov{p}\,ds+a({\bf u},{\bf u})\no\\
&=& -\frac{1}{\eta}b(p,p)+a({\bf u},{\bf u}),
\en
where
\be
\label{Eq:sesquilineara}
a({\bf u},{\bf u}) &=& \int_{\Omega_a}{ \left[\lambda|\nabla\cdot{\bf u}|^2+2\mu\,\mathcal{E}({\bf u}):\ov{\mathcal{E}({\bf u})}- \rho\om^{2}|{\bf u}|^2\right] dx},\\
b(p,p) &=& \int_{\Omega}\left(|\nabla p|^2-k^2|p|^2\right)dx,
\en
and
\ben
\mathcal{E}({\bf u})=\frac{1}{2}\left(\nabla{\bf u}+(\nabla{\bf u})^\top\right).
\enn
Here, $\Gamma_a$ is the circle of radius $a$ and center zero enclosing $\Omega$, $\Omega_a$ is the region between $\Gamma$ and $\Gamma_a$. Since $\mbox{Im}\,a({\bf u},{\bf u})=0$ and $\mbox{Im}\,b(p,p)=0$, taking the imaginary part of \eqref{Eq:4.1-relation} gives
\ben
\mbox{Im}\left(\int_{\Gamma_a}{\bf u}\cdot{\bf T}\ov{\bf u}\,ds\right) =0.
\enn
From the radiation condition for ${\bf u}$ we know
\ben
\mbox{Im}\left(\int_{\Gamma_a}{\bf u}\cdot{\bf T}\ov{\bf u}\,ds\right) \rightarrow-\omega\int_{|x|=1}\,|{\bf u}^\infty|^2\,ds\quad\mbox{as}\quad a\rightarrow\infty
\enn
where ${\bf u}^\infty=({\bf u}_p^\infty,{\bf u}_s^\infty)$ is the far-field pattern of the scattered wave ${\bf u}$. Then we conclude that ${\bf u}^\infty=0$ which implies that ${\bf u}=0$ in $\Omega^c$ by Rellich's lemma and unique continuation. In particular, ${\bf u}=0$ and ${\bf Tu}=0$ on $\Gamma$. Hence $p=0$ and $\pa p/\pa n=0$ on $\Gamma$. Then Holmgren's uniqueness theorem implies that $p=0$ in $\Omega$. Consequently, we see that
\ben
\psi_0 = p^{+} - p^{-} = 0, \quad {\bf v}_0 =  {\bf u}^{-} - {\bf u}^{+}=0
\enn
as expected, where we have denoted by $f^- $ and $f^-$ the limits of the function approach to $\Gamma$ from $\Omega$ and $\Omega^c$, respectively.  This completes the proof.
\end{proof}

{Next, we show that the solution of the system \eqref{Eq:IDBIE5}--\eqref{Eq:IDBIE6} exists by considering  its weak formulation. }

\subsection{Weak formulation}
\label{sec:4.2}
{The weak formulation of the system \eqref{Eq:IDBIE5}--\eqref{Eq:IDBIE6}  reads:} {\it Given $p^{inc}$ and $\partial p^{inc}/\partial n$, find $({\bf v},\psi) \in  \mathcal{H}(\Gamma)$ satisfying}
\be
\label{Eq:WeakForm2}
B({\bf v},\psi;{\bf w},\varphi) = G({\bf w},\varphi),\quad \forall\,({\bf w},\phi)\in \mathcal{H}(\Gamma).
\en
The sesquilinear form $B(\cdot\,;\,\cdot): \mathcal{H}(\Gamma)\times \mathcal{H}(\Gamma)\mapsto \R$  is given by
\be
\label{Eq:SesquilinearForm2}
B({\bf v},\psi;{\bf w},\varphi)  & =& \frac{1}{\eta}\langle W_f\psi,\varphi\rangle -\langle{\bf n}\cdot V_s({\bf n}\psi),\varphi\rangle + \langle K_f^{'}({\bf n}\cdot {\bf v}),\psi\rangle - \langle{\bf n}\cdot K_s{\bf v},\varphi\rangle \nonumber\\
& + &   \langle W_s{\bf v}, {\bf w}\rangle - \eta\langle {\bf n}V_f({\bf n}\cdot {\bf v}), {\bf w}\rangle + \langle{\bf n}K_f\psi,{\bf w}\rangle - \langle K_s^{'}({\bf n}\psi),{\bf w}\rangle
\en
and the linear functional $G({\bf w},\varphi)$ on $\mathcal{H}(\Gamma)$ is defined by
\ben
G({\bf w},\varphi) = \langle g_1,{\bf w}\rangle + \left\langle g_2, \varphi\right\rangle.
\enn
In order to show the existence of a weak solution of the variational equation \eqref{Eq:WeakForm2}, we need the next two theorems.
\begin{theorem}
\label{Theorem4.2}
The sesquilinear form \eqref{Eq:SesquilinearForm2} satisfies a G\aa rding's inequality in the form
\be
\label{Eq:Garding2}
\mbox{Re}\,\{ B({\bf v},\psi; {\bf v},\psi)\} &\ge& \alpha\left(\|{\bf v}\|_{(H^{1/2}(\Gamma))^2}^2 + \|\psi\|_{H^{1/2}(\Gamma)}^2\right) \nonumber\\
&-& \beta\left(\|{\bf v}\|_{(H^{1/2-\epsilon}(\Gamma))^2}^2 + \|\psi\|_{H^{1/2-\epsilon}(\Gamma)}^2\right),
\en
for all $({\bf v}, \psi)\in \mathcal{H}(\Gamma)$ where $\alpha>0$, $\beta\ge 0$ and $0<\epsilon<1/2$ are all constants.
\end{theorem}
\begin{proof}
The proof follows strictly that of Theorem \ref{Theorem3.1}. In \eqref{Eq:SesquilinearForm2}, we set $({\bf w},\varphi)$ to be $({\bf v},\psi)$,  and thus obtain
\ben
B({\bf v},\psi;{\bf v},\psi)  & =& \frac{1}{\eta}\langle W_f\psi,\psi\rangle -\langle{\bf n}\cdot V_s({\bf n}\psi),\psi\rangle + \langle K_f^{'}({\bf n}\cdot {\bf v}),\psi\rangle - \langle{\bf n}\cdot K_s{\bf v},\psi\rangle \nonumber\\
& + &   \langle W_s{\bf v}, {\bf v}\rangle - \eta\langle {\bf n}V_f({\bf n}\cdot {\bf v}), {\bf v}\rangle + \langle{\bf n}K_f\psi,{\bf v}\rangle - \langle K_s^{'}({\bf n}\psi),{\bf v}\rangle.
\enn
We first notice that
\be
\label{Eq:Therorem42Eq2}
\mbox{Re}\,\{\langle W_f\psi,\psi\rangle\}&\ge& \alpha_1 \|\psi\|_{H^{1/2}(\Gamma)}^2 - \beta_1\|\psi\|_{H^{1/2-\epsilon}(\Gamma)}^2, \\
\label{Eq:Therorem42Eq3}
\mbox{Re}\,\{\langle W_s{\bf v},{\bf v}\rangle\} &\ge& \alpha_2 \|{\bf v}\|_{(H^{1/2}(\Gamma))^2}^2 - \beta_2\|{\bf v}\|_{H^{1/2-\epsilon}(\Gamma)}^2.
\en
where $\alpha_1,\alpha_2>0$, $\beta_1,\beta_2\ge 0$ and $1/2>\epsilon>0$ are all constants.  Similarly, we also have
\be
\label{Eq:Therorem42Eq4}
\mbox{Re}\,\{-\langle{\bf n}\cdot V_s({\bf n}\psi),\psi\rangle\} &\ge& -\beta_3\|\psi\|_{H^{1/2-\epsilon}(\Gamma)}^2,\\
\label{Eq:Therorem42Eq5}
\mbox{Re}\,\{-\langle {\bf n}V_f({\bf n}\cdot {\bf v}), {\bf v}\rangle \} &\ge& -\beta_4\|{\bf v}\|_{(H^{1/2-\epsilon}(\Gamma))^2}^2,
\en
and
\be
\label{Eq:Therorem42Eq6}
\mbox{Re}\,\left\{\langle K_f^{'}({\bf n}\cdot {\bf v}),\psi\rangle - \langle{\bf n}\cdot K_s{\bf v},\psi\rangle + \langle{\bf n}K_f\psi,{\bf v}\rangle - \langle K_s^{'}({\bf n}\psi),{\bf v}\rangle\right\}\nonumber\\
\ge -\beta_5\left(\|{\bf v}\|_{(H^{1/2-\epsilon}(\Gamma))^2}^2 + \|\psi\|_{H^{1/2-\epsilon}(\Gamma)}^2\right),
\en
where  $\beta_j\ge 0,\,j=3,4,5$ and $1/2>\epsilon>0$ are constants. {The  inequality \eqref{Eq:Garding2} then  follows  immediately from  \eqref{Eq:Therorem42Eq2}--\eqref{Eq:Therorem42Eq6}. }
\end{proof}

\begin{theorem}
\label{Theorem4.3}
If the surface $\Gamma$ and the material parameter $(\mu,\lambda,\rho)$ are such that there are no traction free solutions, then variational equation \eqref{Eq:WeakForm2} has a unique solution $({\bf v},\psi)  \in  \mathcal{H}(\Gamma).$
\end{theorem}

{We note that the uniqueness of variational equation \eqref{Eq:WeakForm2} is an immediate result of  Theorem \ref{Theorem4.1}.  Again, G\aa rding's inequality and the uniqueness  lead to the existence of a weak solution of variational equation \eqref{Eq:WeakForm2}.}
\section{Burton-Miller formulation based on direct method}
\label{sec:5}
In this section, we {apply the} Burton-Miller formulation { to boundary integral
equations from the direct method in order to remove} irregular values of $-k^2$. We also study the  uniqueness and existence of the weak solution for the derived system of boundary integral equations.

\subsection{Boundary integral equations}
\label{sec:5.1}
We now combine the boundary integral equations \eqref{Eq:FluidBIE}--\eqref{Eq:FluidBIE1} together with transmission conditions  \eqref{Eq:TransCond1}--\eqref{Eq:TransCond2} to obtain a system of coupled boundary integral equations for a pair of unknown functions ${\bf u}$ and $p$, i.e.,
\be
W_s{\bf u}(x) + \left( \frac{1}{2}I - K_s^{'}\right)(p{\bf n})(x) &=& f_1,    \label{Eq:BMBIE1}   \\
\eta\left(\frac{1}{2}I + K_f^{'}+\beta V_f\right)({\bf u}\cdot {\bf n})(x) + \left[\beta\left(\frac{1}{2}I - K_f\right)+ W_f\right] p(x)  &=& \widetilde f_2.   \label{Eq:BMBIE2}
\en
where
\ben
\widetilde f_2 &=& \left(\frac{1}{2}I + K_f^{'}+{\beta} V_f\right)\frac{\partial p^{inc}}{\partial n}(x),
\enn
and {$\beta$} is a constant at our disposal. The unique solvability of \eqref{Eq:BMBIE1}--\eqref{Eq:BMBIE2} is given in the following theorem.
\begin{theorem}
\label{Theorem5.11}
If\\
(a) \, the surface $\Gamma$ and the material parameters $(\mu,\lambda,\rho)$ are such that there are no traction free solutions,\\
(b) \, $\mbox{Im}\,\beta\ne 0$, \\
the system of boundary integral equations \eqref{Eq:BMBIE1}--\eqref{Eq:BMBIE2} is uniquely solvable.
\end{theorem}
\begin{proof}
It is sufficient to prove that the corresponding homogeneous system has only the trivial solution. Suppose that $({\bf u}_0,p_0)$ is a solution of the corresponding homogeneous system of \eqref{Eq:BMBIE1}--\eqref{Eq:BMBIE2}. Using the same notations in Theorem \ref{Theorem3.11}, we obtain from the homogeneous form of (\ref{Eq:BMBIE2}) that
\be
\label{511-1}
\frac{\pa p_i}{\pa n}+\beta p_i=0\quad\mbox{on}\quad \Gamma.
\en
Applying Green's second identity to $p_i$ and its complex conjugate $\ov{p_i}$ we obtain
\ben
0 &=& \int_\Omega\left(p_i\Delta\ov{p_i}-\ov{p_i}\Delta p_i\right)dx \\
&=& \int_\Gamma\left(p_i\frac{\pa \ov{p_i}}{\pa n}-\ov{p_i}\frac{\pa p_i}{\pa n}\right)ds\\
&=& 2i\mbox{Im}\,\beta\int_\Gamma|p_i|^2ds.
\enn
Then it follows that $p_i=0$ on $\Gamma$ provided that $\mbox{Im}\,\beta\ne 0$ and equation (\ref{511-1}) yields also $\pa p_i/\pa n=0$ on $\Gamma$. The rest of the proof follows immediately from the same techniques described in Theorem \ref{Theorem3.11}.
\end{proof}

\subsection{Weak formulation}
\label{sec:5.2}
The system of boundary integral equations \eqref{Eq:BMBIE1}--\eqref{Eq:BMBIE2} is converted to its variational formulation which takes the standard form: {\it Given $p^{inc}$ and $\partial p^{inc}/\partial n$, find $({\bf u},p) \in  \mathcal{H}(\Gamma)$ satisfying}
\be
\label{Eq:WeakForm3}
C({\bf u},p;{\bf v},q) = H({\bf v},q),\quad \forall\,({\bf v},q)\in \mathcal{H}(\Gamma).
\en
The sesquilinear form $C({\bf u},p;{\bf v},q): \mathcal{H}(\Gamma)\times \mathcal{H}(\Gamma)\mapsto \R$  is given by
\be  \label{Eq:SesquilinearForm3}
C({\bf u},p;{\bf v},q)  &=& \left\langle W_s{\bf u}, {\bf v}\right\rangle + \left\langle \left( \frac{1}{2}I - K_s^{'}\right)p{\bf n}, {\bf v}\right\rangle\nonumber\\
 &+& \eta\left\langle \left(\frac{1}{2}I + K_f^{'}+\beta V_f\right)({\bf u}\cdot {\bf n}),q\right\rangle + \left\langle \left[\beta\left(\frac{1}{2}I - K_f\right)+ W_f\right]p,q\right\rangle
\en
and the linear functional $H({\bf v},q)$ is defined on $\mathcal{H}(\Gamma)$ by
\ben
H({\bf v},q) = \langle f_1,{\bf v}\rangle + \langle \widetilde f_2, q\rangle.      \label{Eq:LinearFunctional3}
\enn
In order to show the existence of a weak solution of the variational equation \eqref{Eq:WeakForm3}, we need the next theorem.
\begin{theorem}
\label{Theorem5.1}
The sesquilinear form $C({\bf u},p;{\bf v},q)$ satisfies a G{\aa}rding's inequality in the form
\ben
Re\{ C({\bf u},p; {\bf u},p)\} &\ge& \alpha\left(\|{\bf u}\|_{(H^{1/2}(\Gamma))^2}^2 + \|p\|_{H^{1/2}(\Gamma)}^2\right) \nonumber\\
&-& \beta\left(\|{\bf u}\|_{(H^{1/2-\epsilon}(\Gamma))^2}^2 + \|p\|_{H^{1/2-\epsilon}(\Gamma)}^2\right)      \label{Eq:Garding3}
\enn
for all $({\bf u}, p)\in \mathcal{H}(\Gamma)$ where  $\alpha>0$ , $\beta\ge 0$ and $0<\epsilon<1/2$ are all constants.
\end{theorem}
\begin{proof}
It can be proved by following the same arguments in Theorem 4.2.
\end{proof}

Now, the existence result follows immediately from the Fredholm's alternative: uniqueness implies existence. Therefore, we have the following theorem.

\begin{theorem}
\label{Theorem5.3}
The variational equation \eqref{Eq:WeakForm3} admits a unique solution $({\bf u},p)\in  \mathcal{H}(\Gamma)$ {under the same assumptions in Theorem \ref{Theorem5.11}.}
\end{theorem}

\begin{remark}
It can be seen from Theorem 5.2 that with the help of choosing purely imaginary value $\beta$, the condition (b) in Theorem \ref{Theorem3.11} for the uniqueness of direct method can be removed by Burton-Miller formulation. This technique can also be applied for indirect boundary integral equations, see \cite{LM95} for example. In addition, it can be observed that except the Jones frequency, there is no excluded spectrum of eigenvalues for the indirect method proposed in this paper without the help of Burton-Miller formulation. We also numerically discuss the solvability of the proposed three methods in Section 7.
\end{remark}

\section{Regularization formulations for hypersingular boundary integral operators}
\label{sec:6}
{In all the above variational formulations of boundary integral equations, we see that hypersingular boundary integral operators for the Helmholtz equation as well as for the time-harmonic Navier equation are dominated boundary integral operators. They will play a crucial role in the numerical implementation.  From computational point of view, as is well known, it is difficult to obtain accurate numerical approximation for boundary integral operators with highly singular kernels.  For this reason, we now present regularization formulas for these  hypersingular  operators which will allow us to treat  hypersingular kernels in terms of weakly singular kernels instead. These formulas will be employed in our numerical experiments to appear in a forthcoming  communication.}

 {Recall the fundamental solutions of the Helmholtz equation \eqref{Eq:Helmholtz} and of the time-harmonic Navier equation \eqref{Eq:Navier} given respectively  in  \eqref{Eq:HelmholtzFS}
and \eqref{Eq:NavierFS}, namely
\begin{gather*}
\gamma_k(x,y) = \frac{i}{4}H_0^{(1)}(k|x-y|) \\
{\bf E}(x,y) = \frac{1}{\mu}\gamma_{k_s}(x,y){\bf I} + \frac{1}{\rho\omega^2}\nabla_x\nabla_x\left[\gamma_{k_s}(x,y) - \gamma_{k_p}(x,y)\right]
\end{gather*}
with acoustic wave number $k$, and shear and compressional wave numbers denoted by
\ben
k_s = \omega\sqrt{\rho/\mu},\quad k_p = \omega \sqrt{\rho/(\lambda + 2\mu)}.
\enn
}{For} the sake of simplicity throughout this section, {we denote by $R(x,y) =
\gamma_{k_s}(x,y) - \gamma_{k_p}(x,y)$,}
${\bf n}_{x} = (n_{x}^{1} ,n_{x}^{2})^{T}$ the outward unit normal at $x\in\Gamma$, ${\bf t}_{x} =(- n_{x}^{2} ,n_{x}^{1})^{T}$ the tangent vector, $\nabla_{x}=(\partial/\partial x_{1},\partial/\partial x_{2})^{T}$ the gradient operator and $\delta _{ij}$ the Kronecker delta function of $i$ and $j$.

{We begin with the regularization formulation of the hypersingular boundary integral operator $W_f$ associated with the Helmholtz equation.  Similar  to Lemma 1.2.2 in \cite{HW08}, we have the lemma.}
\begin{lemma}\label{lm:6.1}
The operator $W_f$ can be expressed as a composition of tangential derivatives, the outward unit normal and the simple layer potential operator $V_f$ taking the form
\be \label{Eq:FluidRegularization}
W_{f} p(x)=-\frac{d}{ds_x}V_f\left(\frac{dp}{ds}\right)(x)-k^2{\bf n}_x^T V_f(p{\bf n})(x).
\en
\end{lemma}

Next, we present the regularization formula for the hypersingular boundary operator $W_s$ for the time-harmonic Navier equation in the next theorem and postpone the derivation to Appendix A. Here, we set $A=[0,-1;1,0]$.
\begin{theorem} \label{th:main}
The hypersingular boundary integral operator $W_s$ associated with the time-harmonic Navier equation can be expressed as
\be \label{Eq:Ws.1}
W_s{\bf u}(x)&=&\mu k_{s}^2\int_{\Gamma}{ \left\{{\bf n}_{x}{\bf n}_{y}^{T}R(x,y)-{\bf n}_{x}^{T}{\bf n}_{y}{\bf I}\gamma _{k_{s}}(x,y)-A{\bf n}_{x}^{T} {\bf t}_{y}\gamma _{k_{s}}(x,y)\right\} {\bf u}(y)ds_{y} }\nonumber \\
&+& 2\mu^2 k_s^2\int_{\Gamma}{ A{\bf E}(x,y){\bf n}_{x}^{T} {\bf t}_{y} {\bf u}(y)ds_{y}}\nonumber \\
&-&4\mu ^{2}\int_{\Gamma}  {\frac{{d{\bf E}(x,y)}}{{ds_{x}} }\frac{{d{\bf u}(y)}}{{ds_{y}} }ds_{y}}+\frac{{4\mu ^{2}}}{{\lambda + 2\mu} }\int_{\Gamma}  {\frac{{d\gamma
_{k_{p}} (x,y)} }{{ds_{x}} }\frac{{d{\bf u}(y)}}{{ds_{y}} }ds_{y}} \nonumber\\
&+&2\mu \int_{\Gamma}  {{\bf n}_{x} \nabla _{x}^{T} R(x,y)A\frac{{d{\bf u}(y)}}{{ds_{y}} }ds_{y}}\nonumber\\
&-& 2\mu\int_{\Gamma}  {A\nabla_{x} R( x,y){\bf n}_{x}^{T} \frac{{d{\bf u}(y)}}{{d{s_y}}}ds_{y}},
\en
or
\be \label{Eq:Ws.12}
W_s{\bf u}(x)&=&\mu k_{s}^2\int_{\Gamma}{ \left\{{\bf n}_{x}{\bf n}_{y}^{T}R(x,y)-{\bf n}_{x}^{T}{\bf n}_{y}{\bf I}\gamma _{k_{s}}(x,y)+A{\bf n}_{x}^{T} {\bf t}_{y}\gamma _{k_{s}}(x,y)\right\} {\bf u}(y)ds_{y} }\nonumber \\
&-&4\mu ^{2}\int_{\Gamma}  {\frac{{d{\bf E}(x,y)}}{{ds_{x}} }\frac{{d{\bf u}(y)}}{{ds_{y}} }ds_{y}}+\frac{{4\mu ^{2}}}{{\lambda + 2\mu} }\int_{\Gamma}  {\frac{{d\gamma
_{k_{p}}(x,y) } }{{ds_{x}} }\frac{{d{\bf u}(y)}}{{ds_{y}} }ds_{y}} \nonumber\\
&+&2\mu \int_{\Gamma}  {{\bf n}_{x} \nabla _{x}^{T} R(x,y)A\frac{{d{\bf u}(y)}}{{ds_{y}} }ds_{y}}\nonumber\\
&+& 2\mu\int_{\Gamma}  {A\frac{d}{ds_x}(\nabla_{y} R( x,y)){\bf n}_{y}^{T}{\bf u}(y)ds_{y}}.
\en
\end{theorem}

Based on the regularization formulations {presented in Lemma \ref{lm:6.1} and Theorem \ref{th:main}}, {an integration by parts yields} for all ${\bf u}, {\bf v}\in(H^{1/2} (\Gamma))^2$, $p, q\in H^{1/2}(\Gamma)$ {that}
\ben
\langle W_{f} p,q\rangle=\int_\Gamma\int_\Gamma\gamma_k(x,y)\frac{dp(y)}{ds_y} \frac{d\ov{q}(x)}{ds_x}ds_yds_x -k^2 \int_\Gamma\int_\Gamma {\bf n}_x^T{\bf n}_y \gamma_k(x,y)p(y)\ov{q}(x)ds_yds_x,
\enn
\ben
\langle W_s{\bf u},{\bf v}\rangle &=&\mu k_{s}^2\int_{\Gamma}\int_{\Gamma}{ \ov{\bf v}(x)\cdot\left\{\left[{\bf n}_{x}{\bf n}_{y}^{T}R-{\bf n}_{x}^{T}{\bf n}_{y}{\bf I}\gamma _{k_{s}}-A{\bf n}_{x}^{T} {\bf t}_{y}\gamma _{k_{s}}\right] {\bf u}(y)\right\}ds_{y}ds_x }\nonumber \\
&+& 2\mu^2 k_s^2\int_{\Gamma}\int_{\Gamma}{ \ov{\bf v}(x)\cdot\left\{A{\bf E}(x,y){\bf n}_{x}^{T} {\bf t}_{y} {\bf u}(y)\right\}ds_{x}ds_{y}}\nonumber \\
&+&4\mu ^{2}\int_{\Gamma}\int_{\Gamma} \frac{d\ov{\bf v}(x)}{ds_x}\cdot {\left[{\bf E}(x,y)\frac{{d{\bf u}(y)}}{{ds_{y}} }\right]ds_{y}ds_x}\nonumber \\
&-&\frac{{4\mu ^{2}}}{{\lambda + 2\mu} }\int_{\Gamma}\int_{\Gamma} \frac{d\ov{\bf v}(x)}{ds_x}\cdot  {\left[\gamma_{k_p}(x,y)\frac{{d{\bf u}(y)}}{{ds_{y}} }\right]ds_{y}ds_x} \nonumber\\
&+&2\mu \int_{\Gamma}\int_{\Gamma} \ov{\bf v}(x)\cdot  {\left\{{\bf n}_{x} \nabla _{x}^{T} R(x,y)A\frac{{d{\bf u}(y)}}{{ds_{y}} }\right\}ds_{y}ds_x}\nonumber\\
&-& 2\mu\int_{\Gamma}\int_{\Gamma}\ov{\bf v}(x)\cdot  {\left\{A\nabla_{x} R( x,y){\bf n}_{x}^{T} \frac{{d{\bf u}(y)}}{{d{s_y}}}\right\}ds_{y}ds_x},
\enn
or
\ben
\langle W_s{\bf u},{\bf v}\rangle &=&\mu k_{s}^2\int_{\Gamma}\int_{\Gamma}{ \ov{\bf v}(x)\cdot\left\{\left[{\bf n}_{x}{\bf n}_{y}^{T}R-{\bf n}_{x}^{T}{\bf n}_{y}{\bf I}\gamma _{k_{s}}+A{\bf n}_{x}^{T} {\bf t}_{y}\gamma _{k_{s}}\right] {\bf u}(y)\right\}ds_{y}ds_x }\nonumber \\
&+&4\mu ^{2}\int_{\Gamma}\int_{\Gamma} \frac{d\ov{\bf v}(x)}{ds_x}\cdot {\left[{\bf E}(x,y)\frac{{d{\bf u}(y)}}{{ds_{y}} }\right]ds_{y}ds_x}\nonumber \\
&-&\frac{{4\mu ^{2}}}{{\lambda + 2\mu} }\int_{\Gamma}\int_{\Gamma} \frac{d\ov{\bf v}(x)}{ds_x}\cdot  {\left[\gamma_{k_p}(x,y)\frac{{d{\bf u}(y)}}{{ds_{y}} }\right]ds_{y}ds_x} \nonumber\\
&+&2\mu \int_{\Gamma}\int_{\Gamma} \ov{\bf v}(x)\cdot  {\left\{{\bf n}_{x} \nabla _{x}^{T} R(x,y)A\frac{{d{\bf u}(y)}}{{ds_{y}} }\right\}ds_{y}ds_x}\nonumber\\
&-& 2\mu\int_{\Gamma}\int_{\Gamma}\frac{d\ov{\bf v}(x)}{ds_x}\cdot  {\left\{A\nabla_{y} R( x,y){\bf n}_{y}^{T}{\bf u}(y)\right\}ds_{y}ds_x}.
\enn

\section{Galerkin boundary element method}
\label{sec:7}
In this section, we describe the procedure of reducing the Galerkin equation of (\ref{Eq:WeakForm1}), (\ref{Eq:WeakForm2}) and (\ref{Eq:WeakForm3}) to their discrete linear systems of equations.  Consider the direct method for example and let $\mathcal{H}_h$ be a finite dimensional subspace of $\mathcal{H}(\Gamma)$. The Galerkin approximation of (\ref{Eq:WeakForm1}) reads: {\it Given $p^{inc}$ and $\pa p^{inc}/\pa n$, find $({\bf u}_h,p_h)\in\mathcal{H}_h$ satisfying}
\be
A({\bf u}_h,p_h; {\bf v}_h,q_h)  =  F({\bf v}_h,q_h), \quad\forall \,({\bf v}_h,q_h)\in  \mathcal{H}_h.
\label{Eq:Galerkin1}
\en

\begin{theorem}
Suppose that\\
(a) \, the surface $\Gamma$ and the material parameters $(\mu,\lambda,\rho)$ are such that there are no traction free solutions,\\
(b) \, $-k^2$ is not an eigenvalue of the interior Neumann problem for the Laplacian, \\
(c) \, $\mathcal{H}_h$ is a standard boundary element space satisfying the approximation
property. \\
Then there exists a constant $c>0$ independent of $({\bf u},p)$ and $h$ such that the following estimate holds for $t\le s$
\ben
\|({\bf u},p)-({\bf u}_h,p_h)\|_{(H^t(\Gamma))^2\times H^t(\Gamma)} \le ch^{s-t}\|({\bf u},p)\|_{(H^s(\Gamma))^2\times H^s(\Gamma)}.
\enn
\end{theorem}
\begin{proof}
The proof can be completed by introducing the BBL-condition and the analysis of the Galerkin equation given in \cite{HW04}. We omit it here.
\end{proof}

Now we describe briefly a procedure of reducing the Galerkin equation (\ref{Eq:Galerkin1}) to its discrete linear system of equations. Let $x_i, i=1,2,...,N$ be the discretion points on $\G$ and $\G_i$ be the line segment between $x_i$ and $x_{i+1}$. Then the boundary $\G$ is approximated by
\ben
\widetilde{\G}:=\bigcup_{i=1}^N\G_i.
\enn
Let $\{\varphi_i\}, i=1,2,...,N$ be piecewise linear basis functions of $\mathcal{H}_h$. We seek approximation solutions ${\bf U}_h=({\bf u}_h,p_h)$ in the forms
\ben
{\bf u}_h(x)=\sum_{i=1}^N {\bf u}_i\varphi_i(x), \quad p_h(x)=\sum_{i=1}^N p_i\varphi_i(x),
\enn
where ${\bf u}_i\in\C^2$ and $p_i\in\C$ are unknown nodal values of ${\bf u}_h$ and $p_h$ at $x_i$, respectively. The given Cauchy data are interpolated with the form
\ben
p^{inc}(x)=\sum_{i=1}^N p_i^{inc}\varphi_i(x),\quad \nabla p^{inc}=\sum_{i=1}^N {\bf g}_i^{inc}\varphi_i(x),
\enn
where $p_i^{inc}\in\C$ and ${\bf g}_i^{inc}\in\C^2$ are function values of $p^{inc}$ and $\nabla p^{inc}$ at interpolation points. Then Substituting these interpolation forms into (\ref{Eq:Galerkin1}) and setting $\varphi_i, i=1,2,...,N$ as test functions, we arrive at the linear system of equations
\be
\label{reducedsys}
{\bf A}_h\,{\bf X}={\bf F}_h,
\en
where
\ben
{\bf A}_h &=&
\begin{bmatrix}
     {\bf W}_{sh} & \left(\frac{1}{2}{\bf I}_h-{\bf K}_{sph}\right) \\
     \eta\left(\frac{1}{2}{\bf I}_h^\top+{\bf K}_{fph}\right) & {\bf W}_{fh}
\end{bmatrix},\\
{\bf X}&=& \begin{bmatrix}
     {\bf X}_1 \\
     {\bf X}_2
\end{bmatrix},\\
{\bf F}_h&=&\begin{bmatrix}
     \left(-\frac{1}{2}{\bf I}_h+{\bf K}_{sph}\right){\bf b}_1 \\
     \left(\frac{1}{2}{\bf I}_h^\top+{\bf K}_{fph}\right){\bf b}_2
     \end{bmatrix},
\enn
and
\ben
{\bf X}_1 &=& ({\bf u}_1^\top,{\bf u}_2^\top,...,{\bf u}_N^\top)^\top, \\
{\bf X}_2 &=& (p_1,p_2,...,p_N)^\top, \\
{\bf b}_1 &=& (p_1^{inc},p_2^{inc},...,p_N^{inc})^\top, \\
{\bf b}_2 &=& ({\bf g}_1^{inc},{\bf g}_2^{inc},...,{\bf g}_N^{inc})^\top.
\enn
The stiffness matrix ${\bf A}_h$ consists of block matrices with corresponding entries defined by
\ben
{\bf W}_{sh}(i,j) &=& \int_{\widetilde{\G}} (W_s\varphi_j)\varphi_i\,ds,\\
{\bf W}_{fh}(i,j) &=& \int_{\widetilde{\G}} (W_f\varphi_j)\varphi_i\,ds,\\
{\bf K}_{sph}(i,j) &=& \int_{\widetilde{\G}} (K_s'(\varphi_j{\bf n}))\varphi_i\,ds,\\
{\bf K}_{fph}(i,j) &=& \int_{\widetilde{\G}} (K_f'(\varphi_j{\bf n}^\top))\varphi_i\,ds,\\
{\bf I}_h(i,j) &=& \int_{\widetilde{\G}} \varphi_j{\bf n}\varphi_i\,ds.
\enn
For the implementation of the stiffness matrix ${\bf A}_h$, we refer to the numerical strategy described in \cite{BXY} for solving exterior elastic scattering problem. The computational formulations are omitted here. We denote ${\bf B}_h$ and ${\bf C}_h$ the corresponding stiffness matrix for the indirect method and Burton-Miller formulation.

\section{Numerical experiments}
\label{sec:8}

In this section, we present two numerical tests to demonstrate efficiency and accuracy of the presented systems of BIEs, the regularization formulation and the numerical scheme  for solving the fluid-solid interaction problem. Numerical simulations  are performed under the system of Matlab software using a direct solver for corresponding  linear systems.

We first introduce a model problem for which analytical solutions are available for the evaluation of accuracy. We consider the scattering of a plane incident wave $p^{inc}=e^{ikx\cdot d}$ with direction $d=(1,0)$ by a disc-shaped elastic body of radius $R_0$, and thus we could write  the solution of \eqref{Eq:Navier}--\eqref{Eq:RadiationCond} in the forms
\ben
p(r,\theta) &=& \sum\limits_{n = 0}^\infty  {{A_n}H_n^{(1)}(kr)\cos (n\theta )},\\
{\bf u} &=& \nabla \varphi  -  \nabla\times \psi
\enn
with
\ben
\varphi(r,\theta) &=& \sum\limits_{n = 0}^\infty  {{B_n}J_n(k_pr)\cos (n\theta )},\\
\psi(r,\theta) &=& \sum\limits_{n = 0}^\infty  {{C_n}J_n(k_sr)\sin (n\theta )},
\enn
where the coefficients $A_n$, $B_n$ and $C_n$ are to be determined. According to the transmission conditions \eqref{Eq:TransCond1}--\eqref{Eq:TransCond2}, we are able to obtain a linear system of equations as
\ben
{\bf E}_n {{\bf X}_n} = {\bf e}_n ,
\enn
where ${\bf X}_n=(A_n,C_n,D_n)^\top$, the system matrix ${\bf E}_n=\left[ {E_n^{ij}} \right]$ and the right-hand vector ${\bf e}_n=\left[ {e_n^j} \right]$, $ i,j=1,2,3. $ Their elements (identified by the super-script)  are computed using the following formulations

\ben
E_n^{11} &=& -H_{n-1}^{(1)}(kR_0)+\frac{n}{kR_0}H_{n}^{(1)}(kR_0),\\
E_n^{12} &=& \frac{\rho_f\omega^2k_p}{k}\left[J_{n-1}(k_pR_0)-\frac{n}{k_pR_0}J_{n}(k_pR_0)\right],\\
E_n^{13} &=& \frac{\rho_f\omega^2n}{kR_0}J_{n}(k_sR_0),\\
E_n^{21} &=& 0,\\
E_n^{22} &=& \frac{2\mu nk_p}{R_0}J_{n-1}(k_pR_0)-\frac{2\mu (n^2+n)}{R_0^2}J_{n}(k_pR_0),\\
E_n^{23} &=& \frac{2\mu (n^2+n)-\mu k_s^2R_0^2}{R_0^2}J_{n}(k_sR_0)-\frac{2\mu k_s}{R_0}J_{n-1}(k_sR_0),\\
E_n^{31} &=& H_{n}^{(1)}(kR_0),\\
E_n^{32} &=& \frac{2\mu (n^2+n)-\mu k_s^2R_0^2}{R_0^2}J_{n}(k_pR_0)-\frac{2\mu k_p}{R_0}J_{n-1}(k_pR_0),\\
E_n^{33} &=& \frac{2\mu nk_s}{R_0}J_{n-1}(k_sR_0)-\frac{2\mu (n^2+n)}{R_0^2}J_{n}(k_sR_0),
\enn
and
\ben
e_n^{1} &=& \epsilon_ni^n\left[J_{n-1}(kR_0)-\frac{n}{kR_0}J_{n}(kR_0) \right],\\
e_n^{2} &=& 0,\\
e_n^{1} &=& -\epsilon_ni^nJ_{n}(kR_0).
\enn
For this model problem, the Jones frequencies can be determined by the zeros of $|\mbox{det}\,{\bf E}_n|$. In addition, the eigenvalues of the interior Neumann problem for the Laplacian are related with the zeros of
\ben
J_n'(kR_0)=-J_{n+1}(kR_0)+\frac{n}{kR_0}J_n(kR_0),\quad n\in\Z.
\enn

{\bf Example 1.} In this example, we test the accuracy of proposed  numerical schemes for solving the two dimensional fluid-solid interaction problem. We consider the above model problem with $R_0=0.01\,\mbox{m}$, and $c_s=3122\,\mbox{m/s}$, $c_p=6198\,\mbox{m/s}$. Here, $c_s$ ans $c_p$ are the wave speeds of shear wave and pressure wave in the solid defined by
\ben
c_s=\sqrt{\frac{\mu}{\rho}},\quad c_p=\sqrt{\frac{\lambda+2\mu}{\rho}}.
\enn
The speed of sound in water is $c=1500\,\mbox{m/s}$, and the density of water is $\rho_f=1000\,\mbox{kg/m}^3$, respectively. The density of aluminum $\rho=2700\,\mbox{kg/m}^3$,  and the frequency $\omega=50\pi\,\mbox{kHz}$. We apply the direct method and Burton-Miller formulation to obtain the numerical solutions $({\bf u}_h,p_h)$ on $\Gamma$ and present the results for $N=64$ in Fig. \ref{Ex1-1} and \ref{Ex1-2}, respectively. It can be seen that the numerical solutions are in a perfect agreement with the exact ones. We also list the numerical errors of two methods in Table \ref{TableExp1-1} and \ref{TableExp1-2}, respectively and these results verify the optimal convergence order
\ben
\|{\bf U}-{\bf U}_h\|_{(L^2(\Gamma))^2\times L^2(\Gamma)}=O(1/N).
\enn
Next, we consider the indirect method by which we first compute the numerical solutions $({\bf v}_h,\psi_h)$ on $\Gamma$, then use the representations \eqref{Eq:IndirectBRF1}--\eqref{Eq:IndirectBRF2} to calculate the numerical solutions ${\bf u}_h$ on $\Gamma_{R_0/2}$ and $p_h$ on $\Gamma_{2R_0}$, where
\ben
\Gamma_r:=\{x\in\R^2: |x|=r\}.
\enn
The exact and numerical solutions are presented in Fig. \ref{Ex1-3} by choosing $N=1024$, showing the perfect agreement with each other. Corresponding numerical errors are listed in Table \ref{TableExp1-3}, verifying the achievement of the optimal order of accuracy.

\begin{figure}[htbp]
\centering
\begin{tabular}{ccc}
\includegraphics[scale=0.3]{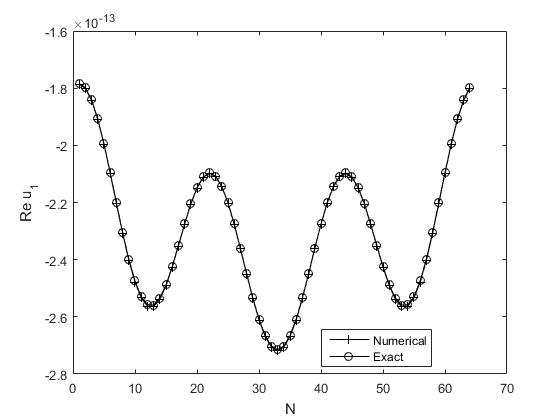} &
\includegraphics[scale=0.3]{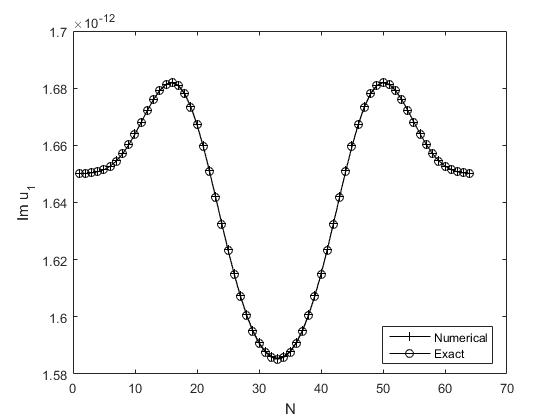} \\
(a) $\mbox{Re}\,u_1$ & (b) $\mbox{Im}\,u_1$ \\
\includegraphics[scale=0.3]{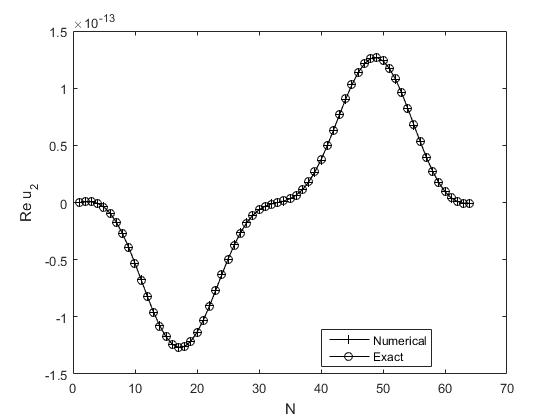} &
\includegraphics[scale=0.3]{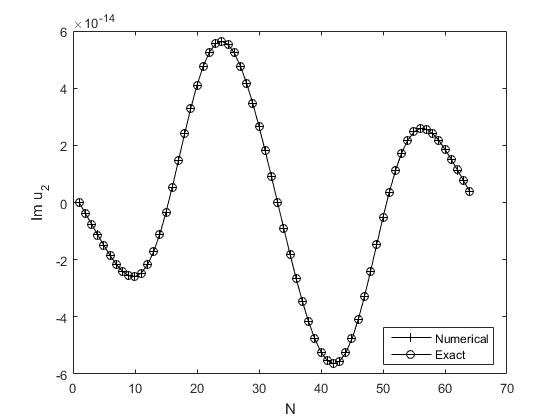} \\
(c) $\mbox{Re}\,u_2$ & (d) $\mbox{Im}\,u_2$ \\
\includegraphics[scale=0.3]{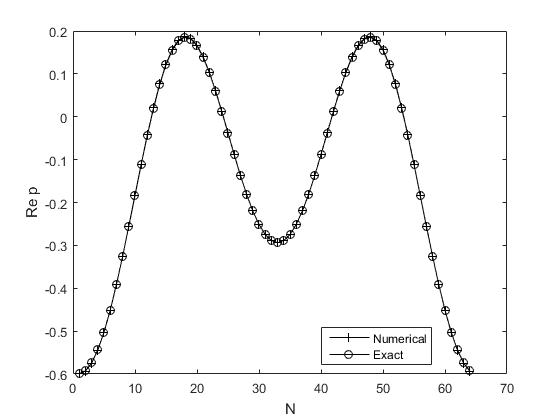} &
\includegraphics[scale=0.3]{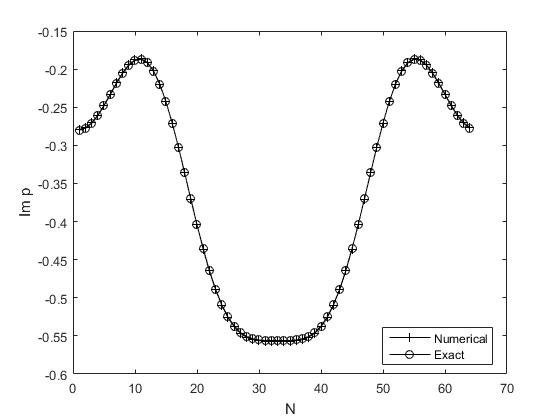} \\
(e) $\mbox{Re}\,p$ & (f) $\mbox{Im}\,p$ \\
\end{tabular}
\caption{Numerical solutions $({\bf u}_h,p_h)$ on $\Gamma$ of the direct method and corresponding exact solutions for Example 1 with $N=64$.}
\label{Ex1-1}
\end{figure}

\begin{table}[htbp]
\caption{Numerical errors of direct method in $L^2$-norm with respect to $N$ for Example 1.}
\centering
\begin{tabular}{|c|c|c|c|c|}\hline
$N$ & $\|{\bf u}-{\bf u}_h\|_{(L^2(\G))^2}$ & Order & $\|p-p_h\|_{L^2(\G)}$ & Order \\
\hline
 64   & 1.35E-16 & --    & 3.10E-4 & --    \\
 128  & 4.66E-17 & 1.53  & 1.24E-4 & 1.32  \\
 256  & 2.35E-17 & 0.99  & 5.95E-5 & 1.06  \\
 512  & 1.24E-17 & 0.92  & 3.01E-5 & 0.98  \\
1024  & 6.24E-18 & 0.99  & 1.53E-5 & 0.98  \\
\hline
\end{tabular}
\label{TableExp1-1}
\end{table}

\begin{figure}[htbp]
\centering
\begin{tabular}{ccc}
\includegraphics[scale=0.3]{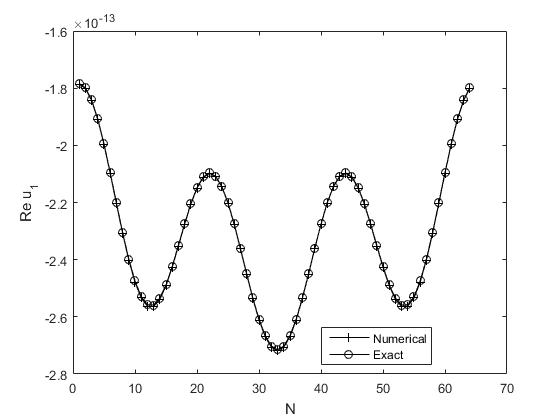} &
\includegraphics[scale=0.3]{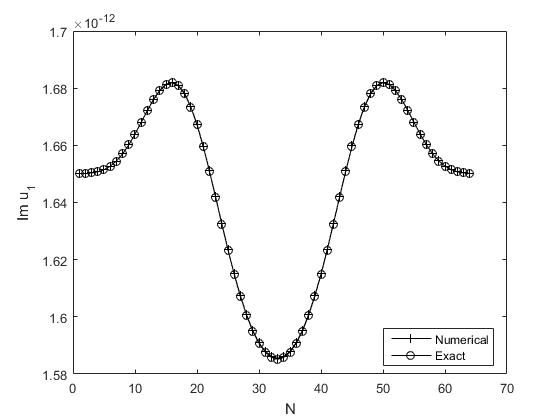} \\
(a) $\mbox{Re}\,u_1$ & (b) $\mbox{Im}\,u_1$ \\
\includegraphics[scale=0.3]{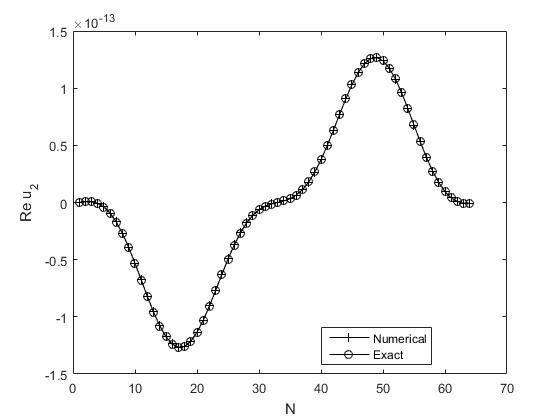} &
\includegraphics[scale=0.3]{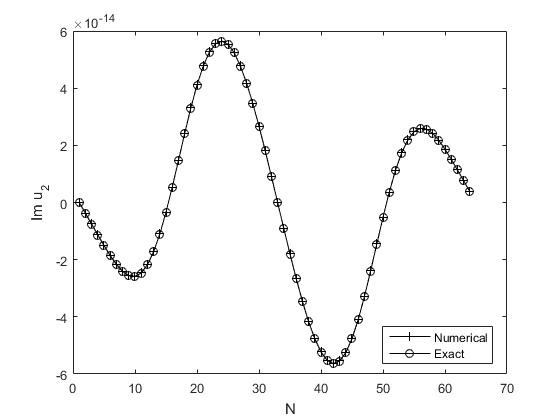} \\
(c) $\mbox{Re}\,u_2$ & (d) $\mbox{Im}\,u_2$ \\
\includegraphics[scale=0.3]{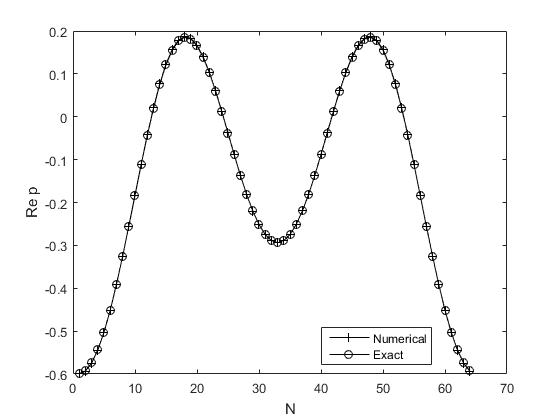} &
\includegraphics[scale=0.3]{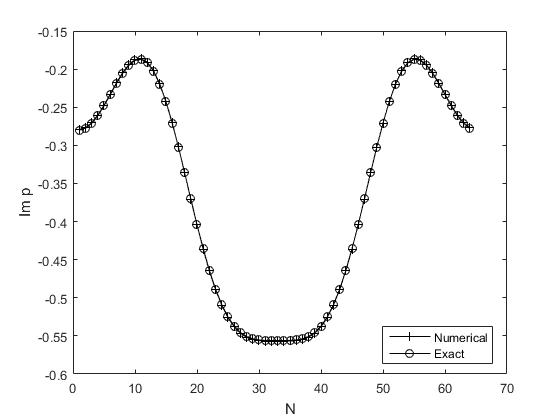} \\
(e) $\mbox{Re}\,p$ & (f) $\mbox{Im}\,p$ \\
\end{tabular}
\caption{Numerical solutions $({\bf u}_h,p_h)$ on $\Gamma$ of the Burton-Miller formulation and corresponding exact solutions for Example 1 with $N=64$.}
\label{Ex1-2}
\end{figure}

\begin{table}[htbp]
\caption{Numerical errors of Burton-Miller formulation in $L^2$-norm with respect to $N$ for Example 1.}
\centering
\begin{tabular}{|c|c|c|c|c|}\hline
$N$ & $\|{\bf u}-{\bf u}_h\|_{(L^2(\G))^2}$ & Order & $\|p-p_h\|_{L^2(\G)}$ & Order \\
\hline
 64   & 1.37E-16 & --    & 3.03E-4 & --    \\
 128  & 5.07E-17 & 1.43  & 1.20E-4 & 1.34  \\
 256  & 2.62E-17 & 0.95  & 5.79E-5 & 1.05  \\
 512  & 1.37E-17 & 0.94  & 2.93E-5 & 0.98  \\
1024  & 6.91E-18 & 0.99  & 1.49E-5 & 0.98  \\
\hline
\end{tabular}
\label{TableExp1-2}
\end{table}

\begin{figure}[htbp]
\centering
\begin{tabular}{ccc}
\includegraphics[scale=0.3]{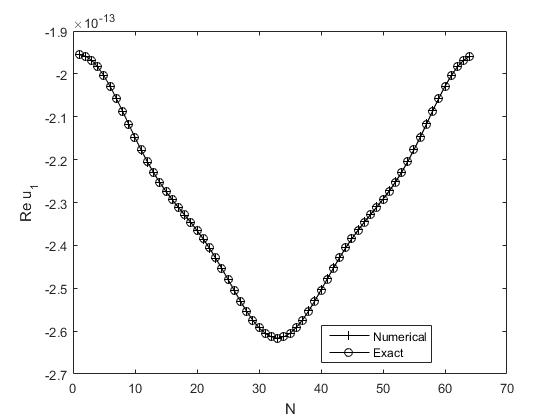} &
\includegraphics[scale=0.3]{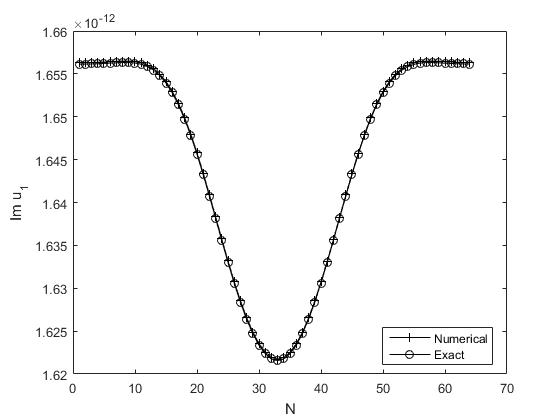} \\
(a) $\mbox{Re}\,u_1$ & (b) $\mbox{Im}\,u_1$ \\
\includegraphics[scale=0.3]{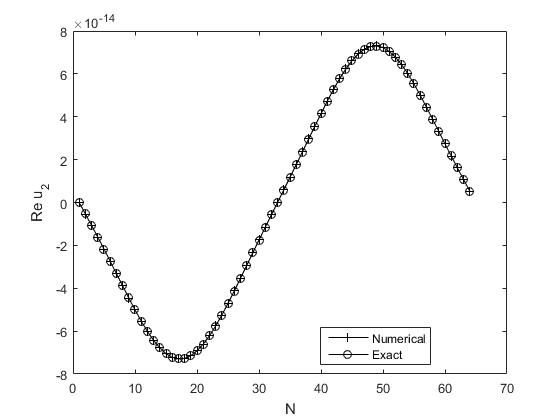} &
\includegraphics[scale=0.3]{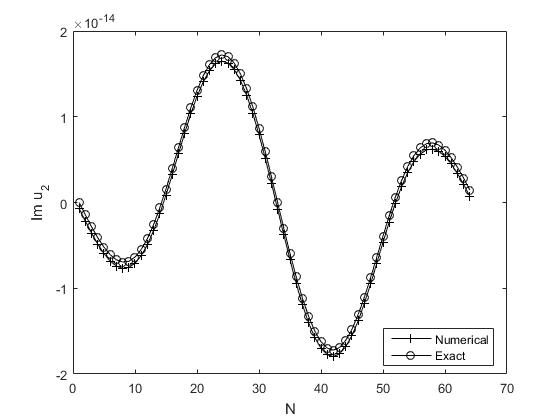} \\
(c) $\mbox{Re}\,u_2$ & (d) $\mbox{Im}\,u_2$ \\
\includegraphics[scale=0.3]{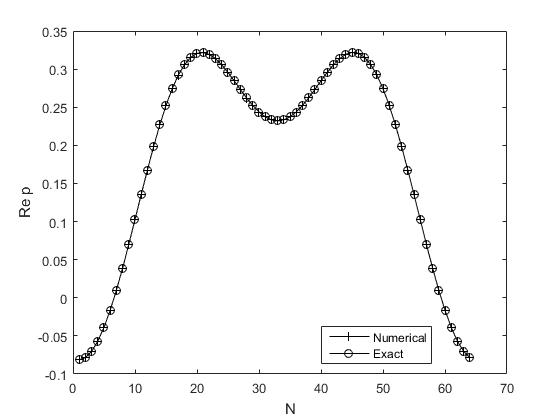} &
\includegraphics[scale=0.3]{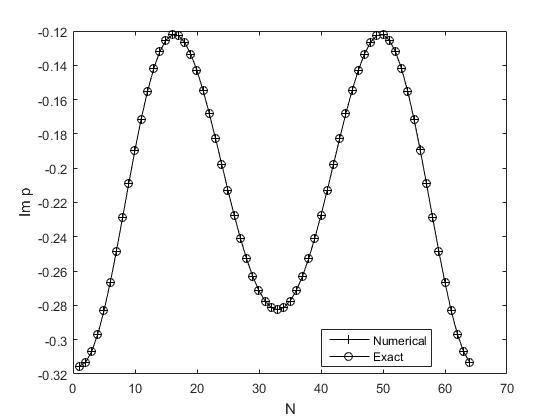} \\
(e) $\mbox{Re}\,p$ & (f) $\mbox{Im}\,p$ \\
\end{tabular}
\caption{Numerical solutions ${\bf u}_h$ on $\Gamma_{R_0/2}$ and $p_h$ on $\Gamma_{2R_0}$ of the indirect method, and corresponding exact solutions for Example 1.}
\label{Ex1-3}
\end{figure}

\begin{table}
\caption{Numerical errors of indirect method in $L^\infty$-norm with respect to $N$ for Example 1.}
\centering
\begin{tabular}{|c|c|c|c|c|}\hline
$N$ & $\|{\bf u}-{\bf u}_h\|_{(L^\infty(\G_{R_0/2}))^2}$ & Order & $\|p-p_h\|_{L^\infty(\G_{2R_0})}$ & Order \\
\hline
 64   & 1.24E-14 & --    & 2.11E-3 & --    \\
 128  & 6.22E-15 & 1.00  & 9.87E-4 & 1.10  \\
 256  & 3.12E-15 & 1.00  & 4.80E-4 & 1.04  \\
 512  & 1.56E-15 & 1.00  & 2.37E-4 & 1.02  \\
1024  & 7.82E-16 & 1.00  & 1.17E-4 & 1.02  \\
\hline
\end{tabular}
\label{TableExp1-3}
\end{table}

{\bf Example 2.} In this example, we demonstrate the occurrence of irregular frequencies  of proposed three systems of BIEs for solving the fluid-solid interaction problem.   We consider the above model problem and choose $\lambda=1$, $\mu=2$, $\rho=1$, $\rho_f=1/2$, $R_0=1$ and $k=\omega$. For $\omega\in[5,10]$, we conclude from the values of $|\mbox{det}\,{\bf E}_n|$ that
\ben
\omega=7.2629
\enn
is the only Jones frequency. Correspondingly, $-k^2$ is an eigenvalue of the interior Neumann problem for the Laplacian when
\ben
\omega &=& 5.3175,\,5.3314,\,6.4156,\,6.7061,\,7.0156,\,7.5013,\\
&\quad& 8.0152,\,8.5363,\,8.5778,\,9.2824,\, 9.6474,\,9.9695.
\enn
Log-log plots of the values $|\mbox{det}\,{\bf A}_h|$, $|\mbox{det}\,{\bf B}_h|$ and $|\mbox{det}\,{\bf C}_h|$ with respect to $\omega$ are presented in Fig. \ref{Ex2-1}, \ref{Ex2-2} and \ref{Ex2-3}, respectively. We can see that the major dips appearing  in these figures are consistent with the theoretical results. For the specified values of $\omega$ denoted using black square in these figures, it can be found that they are the minimum points of $|\mbox{det}\,{\bf E}_n|$.

\begin{figure}[htbp]
\centering
\includegraphics[scale=0.5]{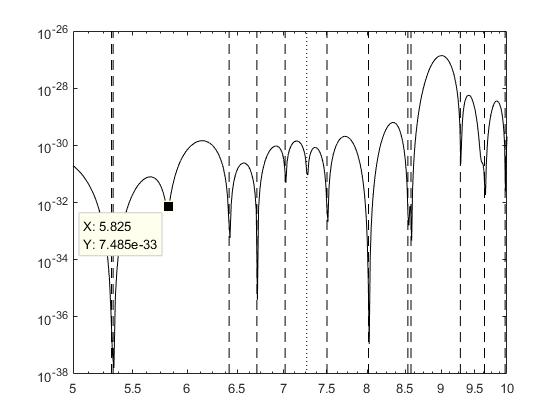}
\caption{The solid line: log-log plot of values $|\mbox{det}\,{\bf A}_h|$ vs. the frequency $\omega$; the vertical dashed line: Neumann eigenvalue; the vertical dotted line: Jones frequency.}
\label{Ex2-1}
\end{figure}

\begin{figure}[htbp]
\centering
\includegraphics[scale=0.5]{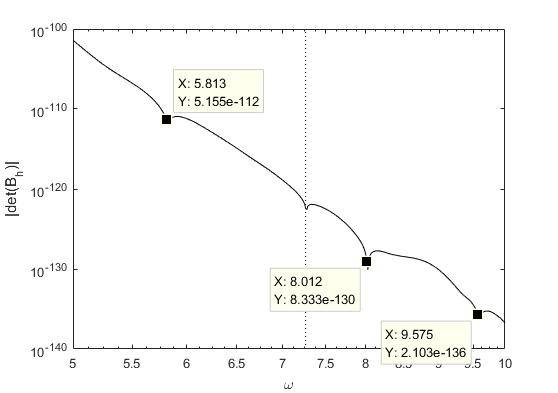}
\caption{The solid line: log-log plot of values $|\mbox{det}\,{\bf B}_h|$ vs. the frequency $\omega$; the vertical dotted line: Jones frequency.}
\label{Ex2-2}
\end{figure}

\begin{figure}[htbp]
\centering
\includegraphics[scale=0.5]{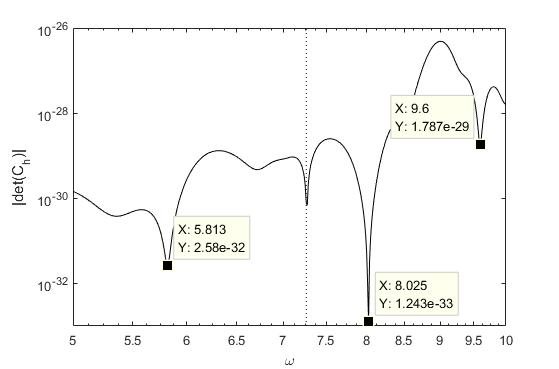}
\caption{The solid line: log-log plot of values $|\mbox{det}\,{\bf C}_h|$ vs. the frequency $\omega$;  the vertical dotted line: Jones frequency.}
\label{Ex2-3}
\end{figure}

\section{Conclusion}
\label{sec:9}
In this paper, through direct and indirect methods, three systems of boundary integral equations are presented for the solution of the two dimensional fluid-solid interaction problems. Uniqueness and existence results have been established for the corresponding variational formulations. These systems can be extended for solving the three-dimensional problem without significant difficulties. A new regularization formulation for the computation of the hyper-singular boundary integral operator associated with the time-harmonic Navier equation in the elastic domain has been derived. Numerical results are presented to validate the regularization formula and the numerical scheme. Applications  of these formulations to inverse problems and eigenvalue problems, and investigations on the preconditioning technique for these systems will envision our future work.
\appendix
\renewcommand\theequation{}
\section{Proof of Theorem \ref{th:main}}
\renewcommand{\theequation}{A.\arabic{equation}}
{For interested  readers, we give the derivations of the regularization formulas \eqref{Eq:Ws.1} and \eqref{Eq:Ws.12} presented in Theorem \ref{th:main} in this appendix.
First we need the following alternative representation of the traction operator}
\begin{lemma} \label{lm:A.1}
The traction operator can be rewritten as
\be \label{Eq:TracationOper1}
{\bf T}_{x}{\bf u} (x)= ( \lambda + \mu){\bf n}_x( \nabla
_{x} \cdot {\bf u}) + \mu \frac{{\partial {\bf u}}}{{\partial n_{x}} } + \mu {\bf M}(
\partial _{x} ,{\bf n}_{x}){\bf u}
\en
where the operator ${\bf M}(\partial _{x} ,{\bf n}_{x})$ is defined by
\ben
{\bf M}(\partial _{x} ,{\bf n}_{x}){\bf u}(x)= \frac{{\partial {\bf u}}}{{\partial
n}_x} + {\bf n}_x \times \nabla\times{\bf u} - {\bf n}_x(\nabla _{x} \cdot {\bf u}).
\enn
Moreover,
\be
{\bf M}( \partial _{x} ,{\bf n}_{x}){\bf u}(x) = A\frac{{d{\bf u}(x)}}{{ds_{x}} }.\label{Eq:M2}
\en
Here, the elements in ${\bf M}(\partial _{x} ,{\bf n}_{x})$ are also called the
{\it G\"unter derivatives} {{\em(\cite{KGBB79})}}.
\end{lemma}

{We also need following preliminary results concerning  applications of  {\it G\"unter derivatives}  to various expressions. }
\begin{theorem} {\label{th:A.2}}
For the fundamental displacement tensor ${\bf E}(x,y)$ defined by \eqref{Eq:NavierFS}, we have
\be \label{Eq:TxExy}
{\bf T}_{x}{\bf E}(x,y) &=& - {\bf n}_{x} \nabla _{x}^{T} R(x,y) + \frac{{\partial \gamma _{k_{s}}(x,y) } }{{\partial n_{x}} }{\bf I}\nonumber\\
&+& {\bf M}( \partial _{x} ,{\bf n}_{x})\left[2\mu {\bf E}(x,y) - \gamma
_{k_{s}}(x,y){\bf I}\right],
\en
where $R(x,y)=\gamma _{k_{s}}(x,y)-\gamma _{k_{p}}(x,y)$.
\end{theorem}
\begin{proof}
Let
\ben
{\bf E}(x) = \frac{1}{\mu}\gamma_{k_s}(x){\bf I} + \frac{1}{\rho\omega^2}\nabla_x\nabla_x\left[\gamma_{k_s}(x) - \gamma_{k_p}(x)\right],
\enn
where
\ben
\gamma_k(x) = \frac{i}{4}H_0^{(1)}(k|x|)
\enn
with $k=k_s$ or $k_p$.
Then it's sufficient to prove that
\be \label{Eq:TEx}
{\bf T}_{x}{\bf E}(x) = - {\bf n}_{x} \nabla _{x}^{T} R(x) + \frac{{\partial \gamma _{k_{s}} (x)} }{{\partial n_{x}} }{\bf I}+ {\bf M}( \partial _{x} ,{\bf n}_{x})[ 2\mu {\bf E}(x) - \gamma
_{k_{s}}(x){\bf I}].
\en
For some matrix $\bf A$ or vector $\bf B$, we denote $({\bf A})_{ij}$ and $({\bf B})_i$ their Cartesian components respectively. Then we have
\be \label{Eq:DIVEx}
&\quad&(\nabla _{x} \cdot {\bf E}(x))_i \nonumber\\
 &=& \frac{1}{\mu }\left[ {\frac{{\partial {\gamma _{{k_s}}(x)}}}{{\partial {x_i}}} + \frac{1}{{k_s^2}}\frac{{{\partial ^3}}}{{\partial x_i^3}}\left( {{\gamma _{{k_s}}(x)} - {\gamma _{{k_p}}(x)}} \right) + \frac{1}{{k_s^2}}\frac{{{\partial ^3}}}{{\partial {x_i}\partial x_j^2}}\left( {{\gamma _{{k_s}}(x)} - {\gamma _{{k_p}}(x)}} \right)} \right] \nonumber\\
&=& \frac{1}{\mu }\left[ {\frac{{\partial {\gamma _{{k_s}}(x)}}}{{\partial {x_i}}} - \frac{1}{{k_s^2}}\frac{\partial }{{\partial {x_i}}}\left( {k_s^2{\gamma _{{k_s}}(x)} - k_p^2{\gamma _{{k_p}}(x)}} \right)} \right] \nonumber\\
&=& \frac{1}{{\lambda  + 2\mu }}\frac{{\partial {\gamma _{{k_p}}(x)}}}{{\partial {x_i}}}
\en
where $i, j=1$ or 2 but $j \ne i$.
Similarly, for $i, j=1$ or 2 we have
\be \label{Eq:NDEx}
{\left( {\frac{{\partial {\bf E}(x)}}{{\partial {n_x}}}} \right)_{ij}} =\frac{1}{\mu } \sum\limits_{l = 1}^2 {n_x^l\left[ {\frac{{\partial {\gamma _{{k_s}}(x)}}}{{\partial {x_l}}}{\delta _{ij}} + \frac{1}{{k_s^2}}\frac{{{\partial ^3}}}{{\partial {x_i}\partial {x_j}\partial {x_l}}}\left( {{\gamma _{{k_s}}(x)} - {\gamma _{{k_p}}(x)}} \right)} \right]}
\en
and
\be \label{Eq:MEx}
&\quad&{\left( {{\bf M}({\partial _x},{n_x}){\bf E}(x)} \right)_{ij}} \nonumber\\
&=& \frac{1}{\mu }\sum\limits_{l = 1}^2 {\left( {\frac{\partial }{{\partial {x_i}}}n_x^l - \frac{\partial }{{\partial {x_l}}}n_x^i} \right)\left[ {{\gamma _{{k_s}}}{\delta _{lj}} + \frac{1}{{k_s^2}}\frac{{{\partial ^2}}}{{\partial {x_j}\partial {x_l}}}\left( {{\gamma _{{k_s}}(x)} - {\gamma _{{k_p}}(x)}} \right)} \right]}.
\en
Therefore, from \eqref{Eq:DIVEx}, \eqref{Eq:NDEx} and \eqref{Eq:MEx} we have
\be \label{Eq:TExij}
{\left( {{{\bf T}_x}{\bf E}(x)} \right)_{ij}}
&=& (\lambda  + \mu ){\left( {{{\bf n}_x}\left( {\nabla  \cdot {\bf E}(x)} \right)} \right)_{ij}} + \mu {\left( {\frac{{\partial {\bf E}(x)}}{{\partial {n_x}}}} \right)_{ij}} - \mu {\left( {{\bf M}({\partial _x},{{\bf n}_x}){\bf E}(x)} \right)_{ij}}\nonumber\\
&+& 2\mu {\left( {{\bf M}({\partial _x},{{\bf n}_x}){\bf E}(x)} \right)_{ij}}\nonumber\\
&=& n_x^i\frac{\partial }{{\partial {x_j}}}\left[ {\frac{{\lambda  + \mu }}{{\lambda  + 2\mu }}{\gamma _{{k_p}}(x)} + \frac{1}{{k_s^2}}\sum\limits_{l = 1}^2 {\frac{{{\partial ^2}}}{{\partial x_l^2}}\left( {{\gamma _{{k_s}}(x)} - {\gamma _{{k_p}}(x)}} \right)} } \right] + \frac{{\partial {\gamma _{{k_s}}(x)}}}{{\partial {n_x}}}{\delta _{ij}}\nonumber\\
&+& {\left( {{\bf M}{{({\partial _x},{{\bf n}_x})}^T}{\gamma _{{k_s}}}(x)} \right)_{ij}}+2\mu {\left( {{\bf M}({\partial _x},{{\bf n}_x}){\bf E}(x)} \right)_{ij}}\nonumber\\
&=& n_x^i\frac{\partial }{{\partial {x_j}}}\left[ {\frac{{\lambda  + \mu }}{{\lambda  + 2\mu }}{\gamma _{{k_p}}(x)} - \frac{1}{{k_s^2}}\left( {k_s^2{\gamma _{{k_s}}(x)} - k_p^2{\gamma _{{k_p}}(x)}} \right)} \right]+\frac{{\partial {\gamma _{{k_s}}(x)}}}{{\partial {n_x}}}{\delta _{ij}}\nonumber\\
&-& {\left( {{\bf M}{{({\partial _x},{{\bf n}_x})}}{\gamma _{{k_s}}}(x)} \right)_{ij}}+2\mu {\left( {{\bf M}({\partial _x},{{\bf n}_x}){\bf E}(x)} \right)_{ij}}\nonumber\\
&=& -n_x^i\frac{\partial R(x)}{{\partial {x_j}}}+\frac{{\partial {\gamma _{{k_s}}}}}{{\partial {n_x}}}{\delta _{ij}}- {\left( {{\bf M}{{({\partial _x},{{\bf n}_x})}}{\gamma _{{k_s}}(x)}(x)} \right)_{ij}}\nonumber\\
&+& 2\mu {\left( {{\bf M}({\partial _x},{{\bf n}_x}){\bf E}(x)} \right)_{ij}}.
\en
and this further leads to \eqref{Eq:TEx} which completes the proof.
\end{proof}
\begin{theorem} \label{th:A.3}
The operator $K_s$ can be expressed as
\be
 K_{s}{\bf u}(x)&=& \int_{\Gamma}  {\frac{{\partial \gamma _{k_{s}}( x,y) } }{{\partial n_{y}
}}{\bf u}( y)ds_{y}}  - \int_{\Gamma}  {\nabla_{y} R( x,y){\bf n}_{y}^{T} {\bf u}(y)ds_{y}} \nonumber\\
&+&\int_{\Gamma}  {[2\mu {\bf E}( x,y) - \gamma _{k_{s}}  (
x,y){\bf I}]{\bf M}(\partial _{y} ,{\bf n}_{y}){\bf u}( y)ds_{y}}.   \label{Eq:Ks}
\en
\end{theorem}
\begin{proof}
As a corollary of Theorem \ref{th:A.2}, we have
\ben
{\bf T}_{y}{\bf E}(x,y) &=& - {\bf n}_{y} \nabla _{y}^{T} R(x,y) + \frac{{\partial \gamma _{k_{s}} (x,y)} }{{\partial n_{y}} }{\bf I}+ {\bf M}( \partial _{y} ,{\bf n}_{y})\left[2\mu {\bf E}(x,y) - \gamma
_{k_{s}}(x,y){\bf I}\right].
\enn
Therefore,
\be\label{Eq:Ksu}
 K_{s}{\bf u}(x)&=& \int_{\Gamma}  {\frac{{\partial \gamma _{k_{s}}( x,y) } }{{\partial n_{y}
}}{\bf u}( y)ds_{y}}  - \int_{\Gamma}  {\left({\bf n}_{y} \nabla _{y}^{T} R(x,y)\right)^{T} {\bf u}(y)ds_{y}}\nonumber\\
&+&\int_{\Gamma}  {\left({\bf M}( \partial _{y} ,{\bf n}_{y})\left[2\mu {\bf E}(x,y) - \gamma_{k_{s}}(x,y){\bf I}\right]\right)^T{\bf u}( y)ds_{y}}.
\en
Regarding the last integral in  \eqref{Eq:Ksu}, integration by parts gives
\be\label{Eq:MgamaTu}
\int_\Gamma  {{{\left({{\bf M}({\partial_y},{{\bf n}_y}) {\gamma_{{k_s}}}(x,y)}\right)}^T}{\bf u}(y)d{s_y}}
&=&  - \int_\Gamma {{\bf M}({\partial _y},{{\bf n}_y}){\gamma _{{k_s}}}(x,y){\bf u}(y)d{s_y}}\nonumber\\
&=& - \int_\Gamma  {A\frac{{d{\gamma _{{k_s}}(x,y)}}}{{d{s_y}}}{\bf u}(y)d{s_y}}\nonumber\\
&=&\int_\Gamma  {{\gamma _{{k_s}}}(x,y){\bf M}({\partial _y},{{\bf n}_y}){\bf u}(y)d{s_y}}
\en
and similarly,
\be\label{Eq:METu}
\int_\Gamma  {{{\left( {{\bf M}({\partial _y},{{\bf n}_y}){\bf E}(x,y)} \right)}^T}{\bf u}(y)d{s_y}}  = \int_\Gamma  {{\bf E}(x,y){\bf M}({\partial _y},{{\bf n}_y}){\bf u}(y)d{s_y}}.
\en
The proof is hence established by a combination of \eqref{Eq:Ksu}, \eqref{Eq:MgamaTu} and \eqref{Eq:METu}.
\end{proof}

{We are now in a position to complete the proof of Theorem \ref{th:main}}.
We know from Theorem \ref{th:A.3} that
\be
\label{Eq:Ws.2}
W_s {\bf u}(x)&=&-{\bf T}_x K_s {\bf u}(x)\nonumber\\
&=& \int_{\Gamma}  {{\bf T}_x\left(\nabla_{y} R( x,y){\bf n}_{y}^{T} {\bf u}(y)\right)ds_{y}}-\int_{\Gamma}  {{\bf T}_x\left(\frac{{\partial \gamma _{k_{s}}( x,y) } }{{\partial n_{y}}}{\bf u}( y)\right)ds_{y}} \nonumber\\
&-&\int_{\Gamma}  {{\bf T}_x\left((2\mu {\bf E}( x,y) - \gamma _{k_{s}}  (
x,y){\bf I}){\bf M}(\partial _{y} ,{\bf n}_{y}){\bf u}(y)\right)ds_{y}}\nonumber\\
&=& {\bf g}_1(x)-{\bf g}_2(x)-{\bf g}_3(x),
\en
where
\ben
{\bf g}_1(x)=\int_{\Gamma}  {{\bf T}_x\left(\nabla_{y} R( x,y){\bf n}_{y}^{T} {\bf u}(y)\right)ds_{y}},
\enn
\ben
{\bf g}_2(x)=\int_{\Gamma}  {{\bf T}_x\left(\frac{{\partial \gamma _{k_{s}}( x,y) } }{{\partial n_{y}}}{\bf u}( y)\right)ds_{y}}
\enn
and
\ben
{\bf g}_3(x)=\int_{\Gamma}  {{\bf T}_x\left((2\mu {\bf E}( x,y) - \gamma _{k_{s}}  (
x,y){\bf I}){\bf M}(\partial _{y} ,{\bf n}_{y}){\bf u}(y)\right)ds_{y}}.
\enn
Therefore, \eqref{Eq:TracationOper1} implies that
\be
\label{Eq:g1.11}
{\bf g}_1(x)&=& \int_{\Gamma}  {{\bf T}_x\left(\nabla_{y} R( x,y){\bf n}_{y}^{T} {\bf u}(y)\right)ds_{y}}\nonumber\\
&=& -(\lambda+\mu)\int_{\Gamma}{{\bf n}_x{{\bf n}_y}^T\Delta R(x,y){\bf u}(y)ds_{y}}\nonumber\\
&+& 2\mu\int_{\Gamma}  { {\bf M}(\partial _{y} ,{\bf n}_{y})\nabla_{x} R( x,y){\bf n}_{x}^{T} {\bf u}(y)ds_{y}}\nonumber \\
&+& \mu\int_{\Gamma}  { \left(\frac{{\partial}}{{\partial n_{x}} }\left(\nabla_{y} R( x,y)\right)-{\bf M}(\partial _{x} ,{\bf n}_{x})\nabla_{y} R( x,y)\right){\bf n}_{y}^{T}{\bf u}(y)ds_{y}}\nonumber\\
&+& 2\mu\int_{\Gamma}  { \left({\bf M}(\partial _{x} ,{\bf n}_{x})\nabla_{y} R( x,y){\bf n}_{y}^{T}-{\bf M}(\partial _{y} ,{\bf n}_{y})\nabla_{x} R( x,y){\bf n}_{x}^{T}\right){\bf u}(y)ds_{y}},
\en
or
\be
\label{Eq:g1.12}
{\bf g}_1(x)&=& -(\lambda+\mu)\int_{\Gamma}{{\bf n}_x{{\bf n}_y}^T\Delta R(x,y){\bf u}(y)ds_{y}}\nonumber\\
&+& 2\mu\int_{\Gamma}  { {\bf M}(\partial _{x} ,{\bf n}_{x})\nabla_{y} R( x,y){\bf n}_{y}^{T} {\bf u}(y)ds_{y}}\nonumber \\
&+& \mu\int_{\Gamma}  { \left(\frac{{\partial}}{{\partial n_{x}} }\left(\nabla_{y} R( x,y)\right)-{\bf M}(\partial _{x} ,{\bf n}_{x})\nabla_{y} R( x,y)\right){\bf n}_{y}^{T}{\bf u}(y)ds_{y}}.
\en
We are able to show that
\ben
\frac{{\partial}}{{\partial n_{x}} }\left(\nabla_{y} R( x,y)\right)-{\bf M}(\partial _{x} ,{\bf n}_{x})\nabla_{y} R( x,y)=-\Delta R(x,y){\bf n}_{x}
\enn
and
\ben
&\quad& {\bf M}(\partial _{x} ,{\bf n}_{x})\nabla_{y} R( x,y){\bf n}_{y}^{T}-{\bf M}(\partial _{y} ,{\bf n}_{y})\nabla_{x} R( x,y){\bf n}_{x}^{T}\\
&=& A\mu k_s^2\left({\bf E}(x,y)-\frac{1}{\mu}\gamma_{k_s}(x,y){\bf I} \right){\bf n}_{x}^{T} {\bf t}_{y},
\enn
and these further yield
\be
\label{Eq:g1.2}
{\bf g}_1(x)&=&(\lambda+2\mu)\int_{\Gamma}{\left[ {k_s^2{\gamma _{{k_s}}}(x,y) - k_p^2{\gamma _{{k_p}}}}(x,y) \right]{\bf n}_x{{\bf n}_y^T}{\bf u}(y)ds_{y}} \nonumber \\
&+& 2\mu^2 k_s^2\int_{\Gamma}{ \left[ {\begin{array}{*{20}{c}}
0&{ - 1}\\
1&0
\end{array}} \right]\left({\bf E}(x,y)-\frac{1}{\mu}\gamma_{k_s}(x,y){\bf I} \right){\bf n}_{x}^{T} {\bf t}_{y} {\bf u}(y)ds_{y}}\nonumber \\
&+&2\mu\int_{\Gamma}  { {\bf M}(\partial _{y} ,{\bf n}_{y})\nabla_{x} R( x,y){\bf n}_{x}^{T} {\bf u}(y)ds_{y}},
\en
or
\be
\label{Eq:g1.22}
{\bf g}_1(x)&=& (\lambda+2\mu)\int_{\Gamma}{\left[ {k_s^2{\gamma _{{k_s}}}(x,y) - k_p^2{\gamma _{{k_p}}}}(x,y) \right]{\bf n}_x{{\bf n}_y^T}{\bf u}(y)ds_{y}} \nonumber \\
&+& 2\mu\int_{\Gamma}  {{\bf M}(\partial_x,{\bf n}_x)(\nabla_{y} R( x,y)){\bf n}_{y}^{T} {\bf u}(y)ds_{y}}.
\en
Similarly, from \eqref{Eq:FluidRegularization} and \eqref{Eq:TracationOper1} we have
\be
\label{Eq:g2.1}
{\bf g}_2(x)&=&\int_{\Gamma}  {{\bf T}_x\left(\frac{{\partial \gamma _{k_{s}}( x,y) } }{{\partial n_{y}}}{\bf u}( y)\right)ds_{y}}\nonumber\\
&=& (\lambda+\mu)\int_{\Gamma}  {{\bf n}_x\frac{{\partial } }{{\partial n_{y}}} \left({\nabla_{x}}^T\gamma _{k_{s}}( x,y)\right){\bf u}( y)ds_{y}}\nonumber\\
&+& \mu\int_{\Gamma}{\frac{{d \gamma _{k_{s}}(x,y) } }{{ds_x}}\frac{{d {\bf u}(y) } }{{ds_y}}ds_y }+\mu k_s^2\int_{\Gamma}{{\bf n}_{x}^{T}{\bf n}_{y}\gamma _{k_{s}}(x,y){\bf u}(y)ds_y }\nonumber\\
&+& \mu\int_{\Gamma}  { {\bf M}(\partial _{x} ,{\bf n}_{x})\left(\frac{{\partial \gamma _{k_{s}}( x,y) } }{{\partial n_{y}}}\right){\bf u}( y)ds_{y}}.
\en
In addition, from \eqref{Eq:TracationOper1}, there holds
\ben
{\bf g}_3(x)&=& \int_{\Gamma}  {{\bf T}_x\left((2\mu {\bf E}( x,y) - \gamma _{k_{s}}  (
x,y){\bf I}){\bf M}(\partial _{y} ,{\bf n}_{y}){\bf u}(y)\right)ds_{y}} \\
&=& 2\mu\int_{\Gamma}  {{\bf T}_x {\bf E}( x,y){\bf M}(\partial _{y} ,{\bf n}_{y}){\bf u}(y)ds_{y}}\\
&-&(\lambda+\mu)\int_{\Gamma}  {{\bf n}_x \nabla_{x}^T\gamma _{k_{s}}(x,y){\bf M} (\partial _{y} ,{\bf n}_{y}){\bf u}(y)ds_{y}}\\
&-& \mu\int_{\Gamma}  { \frac{{\partial}\gamma _{k_{s}}}{{\partial n_{x}}}(x,y){\bf M}(\partial _{y} ,{\bf n}_{y}){\bf u}(y)ds_{y}}\\
&-& \mu\int_{\Gamma}  { {\bf M}(\partial _{x} ,{\bf n}_{x})\gamma _{k_{s}}(x,y){\bf M}(\partial _{y} ,{\bf n}_{y}){\bf u}(y)ds_{y}}.\\
\enn
Then \eqref{Eq:TxExy} leads to
\be
\label{Eq:g3.2}
{\bf g}_3(x)&=& \mu \int_{\Gamma}  {\frac{{\partial \gamma _{k_{s}}(x,y) } }{{\partial n_{x}
}}{\bf M}(\partial _{y} ,{\bf n}_{y}){\bf u}(y)ds_{y}}\nonumber\\
&-& 2\mu \int_{\Gamma}  {{\bf n}_{x} \nabla _{x}^{T} R(x,y){\bf M}(\partial _{y} ,{\bf n}_{y}){\bf u}(y)ds_{y}} \nonumber \\
&+& 4\mu^2 \int_{\Gamma}  {{\bf M}(\partial _{x} ,{\bf n}_{x}){\bf E}(x,y){\bf M}( \partial _{y} ,{\bf n}_{y}){\bf u}(y)ds_{y}}\nonumber \\
&-& 3\mu\int_{\Gamma}  { {\bf M}(\partial _{x} ,{\bf n}_{x})\gamma _{k_{s}}(x,y){\bf M}(\partial _{y} ,{\bf n}_{y}){\bf u}(y)ds_{y}} \nonumber \\
&-& (\lambda+\mu)\int_{\Gamma}  {{\bf n}_x \nabla_{x}^T\gamma _{k_{s}}(x,y){\bf M} (\partial _{y} ,{\bf n}_{y}){\bf u}(y)ds_{y}}.
\en
Let
\be
\label{Eq:g4}
{\bf g}_4(x)=\int_\Gamma  {{\bf M}({\partial _x},{{\bf n}_x}){\gamma _{{k_s}}}(x,y){\bf M}({\partial _y},{{\bf n}_y}){\bf u}(y)d{s_y}},
\en
\be
\label{Eq:g5}
{\bf g}_5(x)=\int_\Gamma  {{\bf M}({\partial _x},{{\bf n}_x}){\bf E}(x,y){\bf M}({\partial _y},{{\bf n}_y}){\bf u}(y)d{s_y}}.
\en
From \eqref{Eq:M2}, we arrive at
\be
\label{Eq:g4.1}
{\bf g}_4(x)&=& \int_\Gamma  {\frac{{d{\gamma _{{k_s}}(x,y)}}}{{d{s_x}}}A^2\frac{{d{\bf u}(y)}}{{d{s_y}}}d{s_y}}= -\int_\Gamma  {\frac{{d{\gamma _{{k_s}}(x,y)}}}{{d{s_x}}}\frac{{d{\bf u}(y)}}{{d{s_y}}}d{s_y}}
\en
\be
\label{Eq:g5.1}
{\bf g}_5(x)&=& \int_\Gamma  {\frac{d}{{d{s_x}}}A\left[ {\frac{1}{\mu }{\gamma _{{k_s}}}(x,y)I + \frac{1}{{\rho {\omega ^2}}}{\nabla _y}{\nabla _y}R(x,y)} \right]A\frac{{d{\bf u}}(y)}{{d{s_y}}}d{s_y}}\nonumber \\
&=& -\frac{1}{\mu }\int_\Gamma  {\frac{{d{\gamma _{{k_s}}(x,y)}}}{{d{s_x}}}\frac{{d{\bf u}(y)}}{{d{s_y}}}d{s_y}}\nonumber \\
&+& \frac{1}{{\mu k_s^2}}\int_\Gamma  {\frac{d}{{d{s_x}}}\left( {{\nabla _y}{\nabla _y}R(x,y) - {\Delta _y}R(x,y)} \right)\frac{{d{\bf u}(y)}}{{d{s_y}}}d{s_y}}\nonumber \\
&=& \int_\Gamma  {\frac{{d{\bf E}(x,y)}}{{d{s_x}}}\frac{{d{\bf u}(y)}}{{d{s_y}}}d{s_y}}-\frac{2}{\mu }\int_\Gamma  {\frac{{d{\gamma _{{k_s}}(x,y)}}}{{d{s_x}}}\frac{{d{\bf u}}}{{d{s_y}}}d{s_y}}\nonumber \\
&+& \frac{1}{{\mu k_s^2}}\int_\Gamma  {\frac{d}{{d{s_x}}}\left( {k_s^2{\gamma _{{k_s}}}(x,y) - k_p^2{\gamma _{{k_p}}}}(x,y)\right)\frac{{d{\bf u}(y)}}{{d{s_y}}}d{s_y}}\nonumber \\
&=& \int_\Gamma  {\frac{{d{\bf E}(x,y)}}{{d{s_x}}}\frac{{d{\bf u}(y)}}{{d{s_y}}}d{s_y}}-\frac{1}{\lambda+2\mu }\int_\Gamma  {\frac{{d{\gamma _{{k_p}}(x,y)}}}{{d{s_x}}}\frac{{d{\bf u}(y)}}{{d{s_y}}}d{s_y}}\nonumber \\
&-& \frac{1}{\mu }\int_\Gamma  {\frac{{d{\gamma _{{k_s}}(x,y)}}}{{d{s_x}}}\frac{{d{\bf u}(y)}}{{d{s_y}}}d{s_y}}.
\en
Therefore, combining \eqref{Eq:g3.2}, \eqref{Eq:g4.1} and \eqref{Eq:g5.1} yields
\be
\label{Eq:g3.3}
{\bf g}_3(x)&=& \mu \int_{\Gamma}  {\frac{{\partial \gamma _{k_{s}} (x,y)} }{{\partial n_{x}
}}{\bf M}(\partial _{y} ,{\bf n}_{y}){\bf u}(y)ds_{y}}
- 2\mu \int_{\Gamma}  {{\bf n}_{x} \nabla _{x}^{T} R(x,y){\bf M}(\partial _{y} ,{\bf n}_{y}){\bf u}(y)ds_{y}}
\nonumber\\
&+& 4\mu^2\int_\Gamma  {\frac{{d{\bf E}(x,y)}}{{d{s_x}}}\frac{{d{\bf u}(y)}}{{d{s_y}}}d{s_y}}-\mu\int_\Gamma  {\frac{{d{\gamma _{{k_s}}(x,y)}}}{{d{s_x}}}\frac{{d{\bf u}(y)}}{{d{s_y}}}d{s_y}}\nonumber \\
&-& \frac{4\mu^2}{\lambda+2\mu }\int_\Gamma  {\frac{{d{\gamma _{{k_p}}(x,y)}}}{{d{s_x}}}\frac{{du}(y)}{{d{s_y}}}d{s_y}}\nonumber\\
&-& (\lambda+\mu)\int_{\Gamma}  {{\bf n}_x \nabla_{x}^T\gamma _{k_{s}}(x,y){\bf M} (\partial _{y} ,{\bf n}_{y}){\bf u}(y)ds_{y}}.
\en
Additionally, let
\ben
{\bf h}_{1}(x) = \int_{\Gamma}  {\left[{\bf M}(\partial _{x} ,{\bf n}_{x})\frac{{\partial \gamma _{k_{s}}(x,y) } }{{\partial n_{y}} }{\bf u}(y) + \frac{{\partial \gamma _{k_{s}}(x,y) } }{{\partial n_{x}} }{\bf M}(\partial _{y} ,{\bf n}_{y}){\bf u}(y)\right]ds_{y}},
\enn
\ben
{\bf h}_{2}(x) = \int_{\Gamma}  {{\bf n}_{x}\left[\frac{{\partial
}}{{\partial n_{y}} }(\nabla _{x}^{T} \gamma _{k_{s}}(x,y)){\bf u}(y) - \nabla _{x}^{T} \gamma _{k_{s}}(x,y){\bf M}(\partial _{y} ,{\bf n}_{y}){\bf u}(y)\right]ds_{y}}.
\enn
Due to integration by parts, we obtain
\be\label{Eq:h1.1}
{\bf h}_{1}(x) &=& \int_{\Gamma}  {\left[{\bf M}(\partial _{x} ,{\bf n}_{x})\frac{{\partial \gamma _{k_{s}}(x,y) } }{{\partial n_{y}} }-{\bf M}(\partial _{y} ,{\bf n}_{y})\frac{{\partial \gamma _{k_{s}}(x,y) } }{{\partial n_{x}} }\right]{\bf u}(y)ds_{y}}  \nonumber\\
&=& - k_{s}^{2} \int_{\Gamma}  {\gamma _{k_{s}}(x,y)A{\bf n}_{x}^{T} {\bf t}_{y}{\bf u}(y)ds_{y}}
\en
and
\be \label{Eq:h2.1}
{\bf h}_{2}(x) &=& \int_{\Gamma}{{\bf n}_{x}\left(\frac{{\partial
}}{{\partial n_{y}} }(\nabla _{x}^{T} \gamma _{k_{s}}(x,y)){\bf u}(y) + \frac{{d}}{{ds_{y}} }(\nabla _{x}^{T} \gamma _{k_{s}}(x,y))A{\bf u}(y)\right)ds_{y}} \nonumber \\
&=& k_{s}^{2} \int_{\Gamma}  {\gamma _{k_{s}}(x,y){\bf n}_{x}{\bf n}_{y}^{T} {\bf u}(y)ds_{y}}.
\en
Hence the proof of \eqref{Eq:Ws.1} is completed by following a combination of \eqref{Eq:g1.2}, \eqref{Eq:g2.1}, \eqref{Eq:g3.3}, \eqref{Eq:h1.1} and \eqref{Eq:h2.1}. In addition, \eqref{Eq:Ws.12} is a consequence of \eqref{Eq:g1.22}, \eqref{Eq:g2.1}, \eqref{Eq:g3.3}, \eqref{Eq:h1.1} and \eqref{Eq:h2.1}.
{\hfill $\Box$}

\section*{Acknowledgments}
The work of T. Yin is partially supported by the NSFC Grant (11371385). The work of L. Xu is partially supported by the NSFC Grant (11371385), the Start-up fund of Youth 1000 plan of China and that of Youth 100 plan of Chongqing University.

\end{document}